\documentclass[11pt,leqno,oneside,letterpaper]{amsart}
\usepackage[letterpaper,left=1.75truein, top=1truein]{geometry}
\usepackage{amsmath,amsfonts,amssymb,amscd,amsthm,amsbsy,epsf,graphics}

\usepackage[T1]{fontenc} % font encoding
\usepackage[utf8]{inputenc}
\usepackage{eulervm} % font

\usepackage{color}
\usepackage[colorlinks,linkcolor=blue,citecolor=blue]{hyperref}
\usepackage{epsfig}
\usepackage{graphicx}
\usepackage{mathtools}
\usepackage[shadow,loadshadowlibrary]{todonotes}

\usepackage{pgf,tikz}
\usepackage[all,cmtip]{xy}

\usetikzlibrary{arrows}
\usetikzlibrary{patterns}
\usetikzlibrary{decorations.markings}

\usepackage{mdframed}

\newmdtheoremenv{testtheorem}{Theorem}

\DeclarePairedDelimiter{\inner}{\langle}{\rangle}

\usepackage{enumitem}
\newlist{casesp}{enumerate}{3} %% new list environment based on enumerate with a max depth of 3
\setlist[casesp]{align=left, %% alignment of labels
 listparindent=\parindent, %% same indentation as in normal text
 parsep=\parskip, %% same parskip as in normal text
 font=\normalfont\bfseries, %% font used for labels
 leftmargin=0pt, %% total amount by which text is indented
 labelwidth=0pt, %% width of labels (=how much they stick out on the left because align=left)
 itemindent=.4em,labelsep=.4em, %% space between label and text
 partopsep=0pt, %% extra vertical space above and below if separate paragraph
 }
\setlist[casesp,1]{label=Case~\arabic*:,ref=\arabic*}
\setlist[casesp,2]{label=Subcase~\thecasespi(\roman*):,ref=\thecasespi.\roman*}
\setlist[casesp,3]{label=Case~\thecasespii.\alph*:,ref=\thecasespii.\alph*}

\newtheorem{theorem}{Theorem}[section]
\newtheorem{lemma}[theorem]{Lemma}
\newtheorem{proposition}[theorem]{Proposition}
\newtheorem{definition}[theorem]{Definition}

\newtheorem{corollary}[theorem]{Corollary}
\theoremstyle{definition} 

\newtheorem{remark}[theorem]{Remark}

\newcommand{\QQ}{\mathbb{Q}}
\newcommand{\ZZ}{\mathbb{Z}}
\newcommand{\RR}{\mathbb{R}}

\newcommand{\NN}{\mathbb{N}}
\newcommand{\LO}{\mathrm{LO}}
\newcommand{\ooo}{\mathfrak{o}}

\newcommand{\LLL}{\mathbb{L}}

\newcommand{\wh}{\mathsf{Wh}}

\title{Order-Detection and Non-left-orderable Surgeries on Links}

\begin{document}

\emergencystretch 3em %attempt to fix line breaks

\author[A.~Clay]{Adam Clay}

\address{Department of Mathematics, 420 Machray Hall, University of Manitoba, Winnipeg, MB, R3T 2N2}
\email{Adam.Clay@umanitoba.ca}

\author[J.~Lu]{Junyu Lu}

\address{Department of Mathematics, 420 Machray Hall, University of Manitoba, Winnipeg, MB, R3T 2N2}
\email{luj9@myumanitoba.ca}

\thanks{Adam Clay
    was partially supported by NSERC grant RGPIN-2020-05343.}

\begin{abstract} Beginning with a $3$-manifold $M$ having a single torus boundary component, there are several computational techniques in the literature that use a presentation of the fundamental group of $M$ to produce infinite families of Dehn fillings of $M$ whose fundamental groups are non-left-orderable. In this manuscript, we show how to use order-detection of slopes to generalise these techniques to manifolds with multiple torus boundary components, and to produce results that are sharper than what can be achieved with traditional techniques alone. As a demonstration, we produce an infinite family of hyperbolic links where many of the manifolds arising from Dehn filling have non-left-orderable fundamental groups. The family includes the Whitehead link, and in that case we produce a collection of non-left-orderable Dehn fillings that precisely matches the prediction of the L-space conjecture.
\end{abstract}

\maketitle

\section{Introduction}

The L-space conjecture posits that an irreducible rational homology $3$-sphere $M$ admits a coorientable taut foliation if and only if $\pi_1(M)$ is left-orderable, and that this happens if and only if $M$ is not an $L$-space. For short, we will often abbreviate each of these properties by saying that $M$ is CTF, $M$ is LO, or $M$ is NLS, respectively.

The behaviour of $L$-spaces with respect to Dehn surgery on a knot in $S^3$ is well understood. If $M$ is the complement of a knot $K$ in $S^3$ and $g(K)$ is the knot genus, then Dehn filling $M$ either produces only NLS manifolds, or it produces NLS manifolds for precisely the Dehn fillings along slopes less than $2g(K)-1$ \cite{OS05}; knots with the latter property are known as L-space knots. Many authors have therefore searched for parallel results describing intervals of slopes for which Dehn filling of $M$ yields a manifold which is CTF, or which is LO (e.g. \cite{CD18, LR14}). %In each case, since the L-space conjecture predicts these structures will be supported by Dehn filled manifolds arising from certain intervals of slopes, the goal has been to develop techniques which handle Dehn fillings along intervals of slopes, rather than working with individual Dehn filled manifolds (e.g. \cite{CD18, LR14}).

%\todo{Add a couple sentences detailing the search for intervals of slopes whose corresponding Dehn filled manifolds have non-left-orderable and left-orderable fundamental group.}

For Dehn fillings of link complements, the situation is complicated by the simultaneous filling of multiple torus boundary components. In analogy with the case of knots, we call an $n$-component link $L$ an L-space link if there exist positive integers $k_1, \ldots, k_n$ such that any Dehn surgery on $L$ with integral surgery coefficients $r_1, \ldots, r_n$ satisfying $$(r_1, \ldots, r_n) \in \prod_{i=1}^n [k_i, \infty)$$ always yields an L-space.
% In analogy with the case of knots, we call a link an L-space link if performing surgery on all components of the link along sufficiently large slopes always yields an L-space.
Focusing on the case of $2$-component links, the pairs of slopes $(r_1, r_2)$ for which the corresponding Dehn filled manifolds are NLS, CTF or LO define regions in the plane. As a particular example, the mirror image of the Whitehead link\footnote{It is our convention that $L5a1$ in the Thistlethwaite table is the Whitehead link.  But we have chosen our setup to agree with \cite{santoro2022spaces}, where $S^3_{*,-1}(\LLL_0)$ is complement of the figure-eight knot, and $S^3_{1,1}(\LLL_0)$ is the Poincaré homology sphere---so we describe our link as the mirror image.} is an L-space link. Santoro established in \cite{santoro2022spaces} that $(r_1,r_2)$-surgery on the mirror image of the Whitehead link produces an NLS manifold if and only if $r_1< 1$ or $r_2< 1$; a CTF manifold exactly when $r_1< 1$ or $r_2< 1$; and when one of $r_1, r_2$ is an integer, a LO manifold precisely when $r_1< 1$ or $r_2< 1$.

Many examples of L-space links also appear in \cite{Liu2017}, and progress towards characterising the possible shapes of regions of surgery coefficients that yield L-spaces can be found in \cite{GORSKY2018386,gorsky2020surgery,Liu2021}. In particular, they studied two-component links having linking number zero or with unknotted components, and showed that the region where pairs of integral surgery coefficients yield L-spaces may not be a product of intervals.

\begin{figure}[h]
    \begin{center}
        \includegraphics[width = 4.5cm]{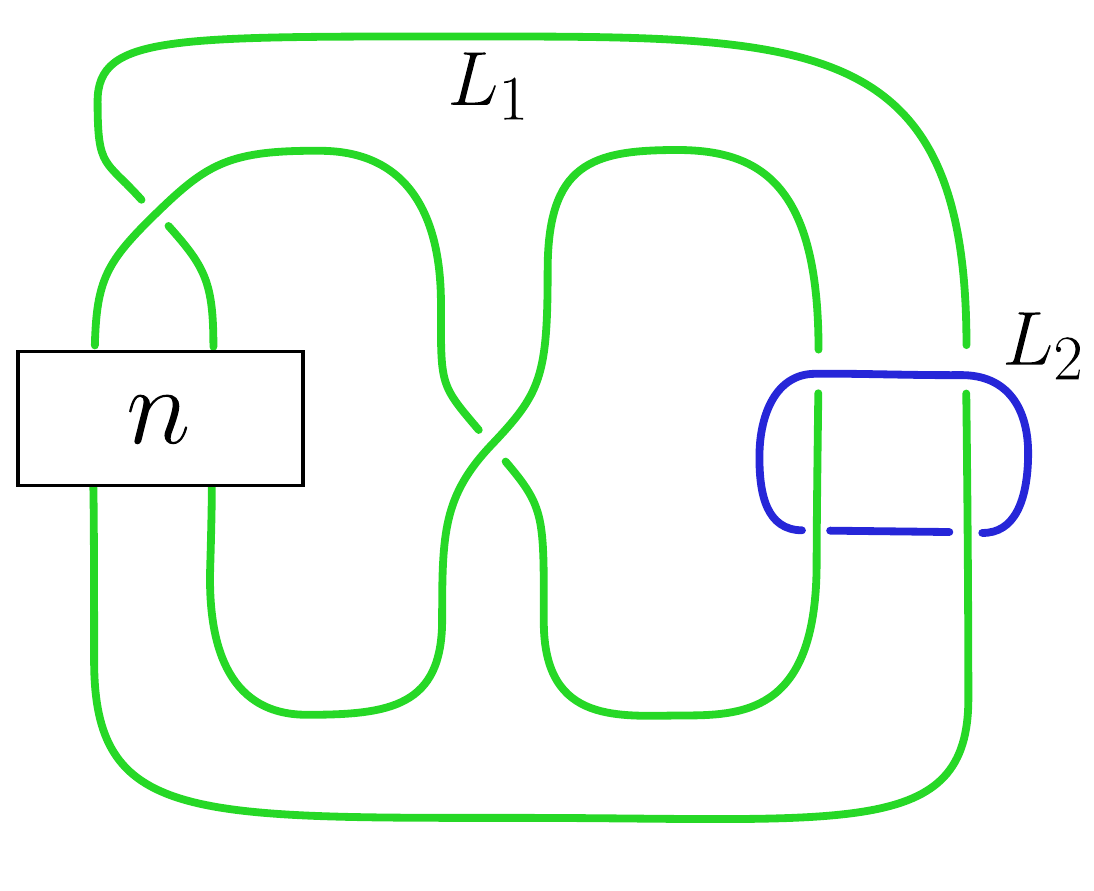}
    \end{center}
    \caption{The two-component link $\LLL_n$; the block represents $n$ full twists}
    \label{fig:linklnn}
\end{figure}

In this manuscript, we use definitions and techniques related to order-detection of slopes as in \cite{BC17, BC23, BGYpreprinta} to generalise the computational techniques of \cite{CW11, CW12} to the case of links. For a $3$-manifold $M$ with torus boundary components $T_1, \ldots, T_n$, order-detection of slopes is a technique developed in \cite{BC17, BC23, claylu2024} for recording the boundary behaviour of left-orderings of $\pi_1(M)$ upon restriction to the peripheral subgroups $\pi_1(T_i)$. It therefore serves to considerably simplify the tracking of boundary behaviour of left-orderings throughout our computations, where we consider the family of links $\LLL_n$ in Figure \ref{fig:linklnn}. We show that:
\begin{theorem}
    \label{thm:introslopes}
    For each integer $n \geq 0$, let $\LLL_n$ denote the link depicted in Figure~\ref{fig:linklnn}. If $(r_1, r_2) \in (2n+2, \infty) \times (2, \infty)$ is a pair of rational numbers, then $(r_1, r_2)$-surgery on $\LLL_n$ produces a manifold which is non-LO.
\end{theorem}

In fact, for $(r_1, r_2) \in (4n+2, \infty) \times (2, \infty)$ we show something stronger, namely that no such pair of slopes is weakly order-detected. See Theorem \ref{thm:lnnnotweak} and Definition \ref{def:order-detect}. Our results for this family of links are not sharp, however, in the following sense: Assuming the truth of the L-space conjecture, \cite[Lemma 3.8]{ding2021} suggests that the manifolds arising from $(r_1, r_2)$-surgery on $\LLL_n$ should be non-LO if and only if $(r_1, r_2) \in [2n+1, \infty) \times [1, \infty)$. With this in mind, we focus on the case $n=0$, which is the mirror of the Whitehead link, and are able to sharpen our results.

\begin{theorem}
    \label{thm:introwhitehead}
    If $(r_1, r_2) \in [1, \infty) \times [1, \infty)$ is a pair of rational numbers, then $(r_1, r_2)$-surgery on the mirror image of the Whitehead link produces a manifold which is non-LO.
\end{theorem}

This result aligns precisely with the non-left-orderability result that is predicted to hold based upon \cite{santoro2022spaces}, and the L-space conjecture. To achieve this sharper result, we show that certain fillings of the Whitehead link obey a property akin to Nie's property (D) \cite{Niepreprint}, see Lemma \ref{lem:whnormalclosure} and Proposition \ref{prop:whprop}. In turn, this allows us to use slope detection results from \cite{BC23}, see Theorem \ref{thm:cofinaltrick}.

\subsection{Organisation of the manuscript} We review background on left-orderings and slope detection in Section \ref{sec:background}. Section \ref{sec:cofinal} contains generalisations of \cite{BC23} to the case of multiple boundary components. We prove Theorems \ref{thm:introslopes} and \ref{thm:introwhitehead} in Section \ref{sec:families}, and conclude that section with observations concerning the behaviour of left-orderings of the figure-eight knot group.

\section{Background}\label{sec:background}

\subsection{Left-orderings}
Let $G$ be a nontrivial group. A left-ordering $\ooo$ of $G$ is defined by a strict total ordering $<_{\ooo}$ of the elements of $G$ such that $g <_\ooo h$ implies $fg <_\ooo fh$ for all $f, g, h \in G$. A left-ordering determines a nonempty set called the \emph{positive cone} $P(\ooo)$, defined by $P(\ooo)=\{g\in G\mid g>_\ooo id\}$. Conversely, a nonempty set $P(\ooo)$ satisfying $G=\{1\}\sqcup P(\ooo)\sqcup P(\ooo)^{-1}$ and $P(\ooo)\cdot P(\ooo)\subset P(\ooo)$ determines a left-ordering $\ooo$ by the prescription that $g<_\ooo h$ if and only if $g^{-1}h\in P(\ooo)$. Elements in $P(\ooo)$ (resp. $P(\ooo)^{-1}$) are said to be \emph{positive} (resp. \emph{negative}). A group is left-orderable if it admits a left-ordering; we adopt the convention that the trivial group is not left-orderable. If $\ooo$ is a left-ordering of $G$ and $H$ is a subgroup of $G$, then $\ooo|_H$ means the restriction of the left-ordering to $H$ in the canonical way.

We denote by $\LO(G)$ the set of all left-orderings of $G$. Identifying each left-ordering with its positive cone, we can view $\LO(G)$ as a subset of the power set $\{0,1\}^G$. Equipping $\{0,1\}^G$ with the product topology, $\LO(G)$ is a closed subset of $\{0,1\}^G$. As $\{0,1\}^G$ is compact, the space $\LO(G)$ is compact, and it is also totally disconnected, Hausdorff, and metrisable when $G$ is countable. There is a $G$-action on $\LO(G)$ by homeomorphisms, defined by $P(g\cdot\ooo)=gP(\ooo)g^{-1}$. See \cite{Sik04, CR16} for more details.

%We see that $h<_{g\cdot \ooo}f$ if and only if $hg<_\ooo fg$ for all $f,g,h\in G$. Therefore, we obtain a group action of $G$ on $\LO(G)$ by left multiplication. The opposite $\ooo^{-1}$ of $\ooo$ is defined by $P(\ooo^{-1})=P(\ooo)^{-1}$. More generally, if we have a monomorphism $\psi:H\to G$ between nontrivial groups, then the left-ordering $\ooo$ on $G$ induces a left-ordering $\psi^{-1}(\ooo)$ on $H$ by $P(\psi^{-1}(\ooo))=\psi^{-1}(P(\ooo))$.

A subgroup $C \subset G$ of a left-orderable group is said to be \emph{convex} with respect to $\ooo$ (or simply $\ooo$-convex) if for all $f\in G$ and $g,h\in C$, $g<_\ooo f<_\ooo h$ implies $f\in C$. This condition is equivalent to the requirement that left cosets $\{gC\mid g\in G\}$ admit a total ordering $\prec$ defined by \[gC \prec hC \iff g <_\ooo h \mbox{ whenever } gC \neq hC.\]
A subgroup $C$ of $G$ is said to be \emph{relatively convex} if it is $\ooo$-convex for some left-ordering $\ooo$. An important example of $\ooo$-convex subgroups arises from lexicographic orderings. Suppose that
\[\xymatrix{\{1\}\ar[r] &N\ar[r]^i&G\ar[r]^p&H\ar[r]&\{1\}}\] is a short exact sequence of nontrivial groups, and that $\ooo_N,\ooo_H$ are left-orderings of $N,H$ respectively. Then $G$ admits a lexicographic ordering $\ooo$ defined by $P(\ooo)=i(P(\ooo_N))\cup p^{-1}(P(\ooo_H)).$ One can check that the subgroup $i(N)$ is $\ooo$-convex.

%For a subset $A$ of $G$, its $\ooo$-convex hull is defined to be $$C(A)=\{g\in G\mid a_1<_\ooo g<_\ooo a_2 \mbox{ for some }a_1,a_2\in A\}.$$ We say a subset $A$ of $G$ is $\ooo$-\emph{cofinal} if $C(A)=G$ and an element $g\in G$ is $\ooo$-cofinal if $C(\inner{g})=G$. 

% In the case of fundamental groups $G=\pi_1(M)$, we say $\ooo$ is $\pi_1(T_i)$-cofinal if the $\ooo$-convex hull $C(\pi_1(T_i))$ is the whole group $\pi_1(M)$. If $n=1$, then we simply say $\ooo$ is boundary cofinal if it is $\pi_1(T_1)$-boundary cofinal.

If $G$ is a countable group, every left-ordering $\ooo$ of $G$ gives rise to an action $\rho_\ooo :G \rightarrow \mathrm{Homeo}_+(\mathbb{R})$ on the real line $\RR$ via the \emph{dynamic realisation}, whose construction is roughly as follows. Fix a left-ordering $\ooo$ of $G$ and choose an order-preserving embedding $t:G\to \RR$ which is \emph{tight}, meaning that for any nontrivial open interval $(a,b)\in \RR\setminus t(G)$, there exist $g,h\in G$ with $g<_\ooo h$ and having no elements of $G$ between them, such that $(a,b)\subset (t(g),t(h))$. Next, define $\rho_\ooo :G \rightarrow \mathrm{Homeo}_+(\mathbb{R})$ in three steps: First, on $t(G) \subset \mathbb{R}$ set $\rho_\ooo(g)(t(h)) = t(gh)$ for all $g, h \in G$, then extend the action continuously to $\overline{t(G)}$, and finally extend affinely across any remaining gaps. This construction of $\rho_\ooo$ depends on the choice of tight embedding $t: G \rightarrow \mathbb{R}$, but is well-defined up to conjugation by elements of $\mathrm{Homeo}_+(\mathbb{R})$. See \cite{BC23} and \cite{Navas10a} for more details.

\subsection{Slope detection and \texorpdfstring{$3$-manifolds}{three-manifolds}}

Unless otherwise indicated, from here forward we will use $M$ to denote a compact, connected, orientable 3-manifold whose boundary is a union of incompressible tori, $\partial M=T_1\cup \dots\cup T_n$. In particular, we call $M$ a knot manifold if $n=1$. When $M$ is the complement of a link $L = L_1 \cup \dots \cup L_n$ in $S^3$, our convention is that $T_i \subset \partial M$ is the boundary torus resulting from the link component $L_i$.

A \emph{slope} on $T_i$ is defined to be an element $[\alpha]$ in $\mathbb{P}H_1(T_i;\RR)$ (the projective space of $H_1(T_i;\RR)$), where $\alpha\in H_1(T_i;\RR)$. To simplify notation, we write $\mathcal{S}(T_i)$ for the set of slopes on $T_i$ and define $\mathcal{S}(M)=\mathcal{S}(T_1)\times \dots \times \mathcal{S}(T_n)$. Identifying $H_1(T_i; \mathbb{Z})$ with the integer lattice points in $H_1(T_i; \mathbb{R})$, we say that a slope $[\alpha]$ is \emph{rational} if $\alpha \in H_1(T_i; \mathbb{Z})$, and \emph{irrational} otherwise. We call a tuple of slopes $([\alpha_1], \dots , [\alpha_n])\in \mathcal{S}(M)$ rational if $[\alpha_i]$ is rational for every $i$; moreover, if $[\alpha]$ is rational, then we always assume that $\alpha$ is a primitive element in $H_1(T_i;\ZZ)=\pi_1(T_i)$. We use $\mathcal{S}_{rat}(M)$ to denote the tuples of rational slopes.

Since the inclusion maps $\pi_1(T_i)\to \pi_1(M)$ are injective, we can identify each group $\pi_1(T_i)$ with a subgroup of $\pi_1(M)$ that is isomorphic to $\ZZ^2$. We fix such an identification for each $i$ and from here forward simply write $\pi_1(T_i)\subset \pi_1(M)$. Note that $S(T_i)$ is homeomorphic to $S^1$, so we may identify $S(T_i)$ with $\RR\cup \{\infty\}$ in a way that identifies the rational slopes with $\mathbb{Q} \cup \{\infty\}$ as follows. Fixing a meridian and longitude pair $(m_i,l_i)$ on $T_i$, for $r=p/q\in \QQ$ in lowest terms, set $\alpha_r=m_i^pl_i^q$ which is viewed as an element in $H_1(T_i; \ZZ)$. Then we make the identification $r\mapsto [\alpha_r]$ and $\infty\mapsto [m_i]$.

When a meridian and longitude basis for $\pi_1(T_i)$ is chosen for all $i \in \{1, \ldots, n\}$, the manifold obtained by performing Dehn surgery along each $T_i$ with slope $r_i \in \QQ \cup \{ \infty \}$ will be denoted by $M(r_1, r_2, \dots, r_n)$. When $M$ is a link complement, namely, $S^3$ with an open regular neighbourhood of the link $L$ removed, we will write $S^3_{r_1,r_2,\dots, r_n}(L)$ instead. For a two-component link $L = L_1 \cup L_2$, we use the notation $S^3_{r_1, *}(L)$ to mean the 3-manifold obtained by performing Dehn filling along $T_1$ with slope $r_1$ while leaving $T_2$ unfilled; similarly, $S^3_{*, r_2}(L)$ denotes the manifold obtained by filling only along $T_2$ with slope $r_2$.

We now return to considering left-orderings, where the case $G=\ZZ^2$ serves an important role. Firstly, observe that for each left-ordering $\ooo$ on $G$, there is a line $L(\ooo)\subset \ZZ^2\otimes \RR=\RR^2$ uniquely determined by the property that all the elements of $\ZZ^2$ lying to one side of it are positive and all the elements lying to the other side are negative; see e.g. \cite[Lemma 3.3]{CR12}. The line $L(\ooo)$ is said to have \emph{rational slope} if $L_0=L(\ooo)\cap \ZZ^2\cong \ZZ$, in which case $L_0$ is $\ooo$-convex. Otherwise, it is said to have \emph{irrational slope}. For a given rational (resp. irrational) slope, there are precisely four (resp. two) left-orderings giving rise to the particular slope. If we give the set of lines through the origin in $\RR^2$ the standard topology and write $[\ell]$ for the image of such a line $\ell\subset \RR^2$ in the resulting copy of $\RR P^1\cong S^1$, then the map $\mathcal{L}:\LO(\ZZ^2)\to \RR P^1$ given by $\mathcal{L}(\ooo)=[L(\ooo)]$ is continuous \cite{BC23}.

%Since $\pi_1(T_i)=\ZZ^2$, a \emph{slope} on $T_i$ is defined to be an element $[\alpha]$ in $\mathbb{P}H_1(T_i;\RR)$ (the projective space of $H_1(T_i;\RR)$), where $\alpha\in H_1(T_i;\RR)$. Since the inclusion maps $\pi_1(T_i)\to \pi_1(M)$ are injective, we can identify each group $\pi_1(T_i)$ with a subgroup of $\pi_1(M)$ that is isomorphic to $\ZZ^2$. We fix such an identification for each $i$ and from here forward simply write $\pi_1(T_i)\subset \pi_1(M)$.

Denote by $r_i: \LO(\pi_1(M))\to\LO(\pi_1(T_i))$ the restriction map of left-orderings. The \emph{slope map} $s:\LO(M)\to \mathcal{S}(M)$ is defined by $s(\ooo)=(\mathcal{L}(r_1(\ooo)),\dots,\mathcal{L}(r_i(\ooo)))$. The slope map $s$ is continuous since $r_i$ is continuous for all $i$, and $\mathcal{L}$ is continuous. %We also write $s_i:\LO(\pi_1(M))\to \mathcal{S}(T_i)$ to be the composition of $s$ with the projection to the $i$-th factor, equivalently, the composition $\mathcal{L}\circ r_i$. 
One can show that this means $[L(r_i(\mathfrak{o}))]$ is rational if $L(r_i(\mathfrak{o})) \cap H_1(T_i; \mathbb{Z}) \cong \mathbb{Z}$; otherwise, the slope is irrational. So the terminology we have introduced concerning rational and irrational slopes is consistent.

\begin{definition}[\cite{claylu2024}]\label{def:order-detect}
    Suppose that $\mathfrak{o}$ is a left-ordering of $\pi_1(M)$, and let $J \subset K \subset \{ 1, \ldots, n\}$ and $([\alpha_1], \ldots, [\alpha_n])\in \mathcal{S}(M)$. We say that $(J, K; [\alpha_1], \ldots, [\alpha_n])$ is \emph{order-detected} by $\ooo$ if
    \begin{enumerate}
        \item[O1.] $s(\mathfrak{o}) = ([\alpha_1], \ldots, [\alpha_n])$;
        \item[O2.] for all $g \in \pi_1(M)$, we have $s(g \cdot \ooo) = ([\beta_1], \ldots, [\beta_n])$ where $[\beta_i] = [\alpha_i]$ for all $i \in K$;
        \item[O3.] there exists an $\mathfrak{o}$-convex normal subgroup $C$ such that for all $i\in \{1,\dots,n\}$ if $[\alpha_i]$ is rational then $\pi_1(T_i)\cap C\leq \langle \alpha_i \rangle $ with $\pi_1(T_i)\cap C = \langle \alpha_i \rangle $ whenever $i \in J$, and if $[\alpha_i]$ is irrational then $\pi_1(T_i)\cap C=\{1\}$.
    \end{enumerate}
\end{definition}

We also say that $(J, K; [\alpha_1], \ldots, [\alpha_n])$ is $\ooo$-detected, or that $\ooo$ order-detects $(J, K; [\alpha_1], \ldots, [\alpha_n])$.  We say that $(J, K; [\alpha_1], \ldots, [\alpha_n])$ is order-detected if it is $\ooo$-detected for some left-ordering $\ooo$ of $\pi_1(M)$.

\begin{remark}
    \label{rem:sub-detection}
    Note that if $(J, K; [\alpha_1], \ldots, [\alpha_n])$ is order-detected and $J' \subset J$, $K'\subset K$ and $J'\subset K'$, then $(J', K'; [\alpha_1], \ldots, [\alpha_n])$
    is also order-detected.
\end{remark}

If $(J, K; [\alpha_1], \ldots, [\alpha_n])$ is order-detected, we say that $[\alpha_i]$ is \emph{weakly order-detected} for each $i=1,\dots,n$; and is \emph{strongly order-detected} if $i \in J$, and is \emph{(regularly) order-detected} if $i \in K$. If $M$ is a knot manifold, the language we have just introduced (strong detection, weak detection, detection) agrees with \cite{BC23}, in the sense that $[\alpha] \in \mathcal{S}(M)$ is weakly detected if $s(\ooo) = [\alpha]$, detected if $s(g \cdot \ooo) = [\alpha]$ for all $g \in \pi_1(M)$, and strongly detected if $[\alpha]$ is irrational or $[\alpha]$ is rational and there is an $\ooo$-convex normal subgroup $C$ such that $C \cap \pi_1(T) = \langle \alpha \rangle$.

The notion of cofinality is strongly related to order-detection of slopes, and plays a central role in many of our arguments. For a subset $A$ of $G$, its $\ooo$-convex hull is defined to be \[C(A)=\{g\in G\mid a_1<_\ooo g<_\ooo a_2 \mbox{ for some }a_1,a_2\in A\}.\] We say a subset $A$ of $G$ is $\ooo$-\emph{cofinal} if $C(A)=G$ and an element $g\in G$ is $\ooo$-\emph{cofinal} if $C(\inner{g})=G$. An essential result, which we use both in its form below and in a more general form adapted to deal with multiple boundary components (see Theorem \ref{thm:generalcofinal}), is the following.

\begin{theorem}[{\cite[Theorem 1.7]{BC23}}] Let $M$ be a knot manifold. If not all the slopes in $S(M)$ are weakly order-detected, then $\pi_1(T)$ is $\ooo$-cofinal for every left-ordering $\ooo$ of $\pi_1(M)$.
\end{theorem}

\section{Cofinality, Dehn filling and slope detection}
\label{sec:cofinal}
We generalise the main cofinality result of \cite[Theorem 1.7]{BC23} to the case of a manifold $M$ with multiple boundary components. Our technique for doing so requires the existence of a convex subgroup $C$ containing one of the peripheral subgroups, and in many cases, boundedness of the peripheral subgroup is enough to produce such a subgroup $C$. Below we show how to do this.

\begin{lemma}
    \label{lem:Cexists}
    Suppose that $M$ is a compact, connected, orientable, irreducible $3$-manifold whose boundary consists of incompressible tori $T_1, \dots ,T_n$. Let $J\subset K\subset \{1,\dots,n\}$, and suppose $\ooo$ is a left-ordering of $\pi_1(M)$ $\ooo$ that order-detects $(J, K; [\alpha_1], \dots, [\alpha_n])$. Given a fixed $j\in\{1,2,\dots,n\}$, if $\pi_1(T_j)$
    %\todo{Actually, we can always assume $j=1$.} 
    is not $\ooo$-cofinal, then there exists a left-ordering $\mathfrak{o}'$ of $\pi_1(M)$ and a proper subgroup $C \subset \pi_1(M)$ such that $C$ is $\mathfrak{o}'$-convex, $\pi_1(T_j) \subset C$, and $\ooo'$ order-detects $(\emptyset, \emptyset; [\beta_1], \ldots, [\beta_n])$ where $[\beta_j]=[\alpha_j]$ and $[\beta_i] = [\alpha_i]$ for all $i \in K$.

    % \begin{itemize}
    % \item $C$ is $\mathfrak{o}'$-convex and $\pi_1(T_j) \subset C$,
    % \item $\ooo'$ detects $(\emptyset, K'; [\beta_1], \ldots, [\beta_n])$ where $[\beta_j]=[\alpha_j]$\todo{I added $[\beta_j]=[\alpha_j]$. Remove it if you think it is not true.} and $[\beta_i] = [\alpha_i]$ for all $i \in K$ and $K' = \{ i \in K \mid \pi_1(T_i) \not\subset C\}.$ \todo{I added $K'$ here and the word proper above}
    % \end{itemize}
\end{lemma}
\begin{proof}
    We follow the proof of \cite[Lemma 5.9]{BC23}. Suppose that $\pi_1(T_j)$ is not $\mathfrak{o}$-cofinal and choose a positive element $\gamma \in \pi_1(T_j)$ that is $\mathfrak{o}|_{\pi_1(T_j)}$-cofinal. Choose a tight order-preserving embedding $t: \pi_1(M) \rightarrow \mathbb{R}$ and use it to construct the dynamic realisation $\rho_{\mathfrak{o}} : \pi_1(M) \rightarrow \mathrm{Homeo}_+(\mathbb{R})$.

    Since $\gamma$ is not $\mathfrak{o}$-cofinal, the limit $\lim_{n} t(\gamma^n) = x_0$ exists. Since $\rho_{\mathfrak{o}} (\gamma)(x_0) =x_0$ and $\gamma$ is $\mathfrak{o}|_{\pi_1(T_j)}$-cofinal, one can show that $\rho_{\mathfrak{o}}(h)(x_0) = x_0$ for all $h \in \pi_1(T_j)$.

    Set $C = \mathrm{Stab}_{\rho_{\mathfrak{o}}}(x_0)$ and so $\pi_1(T_j) \subset C$. Note that $C$ is proper since dynamic realisations do not have global fixed points. Then we can use \cite[Proposition 2.5]{BC23} to construct a left-ordering $\mathfrak{o}'$ such that $C$ is $\ooo'$-convex and $\mathfrak{o}'|_C = \mathfrak{o}|_C$. Namely, we declare $g<_{\ooo'} h$ if $\rho_{\ooo}(g)(x_0) < \rho_{\ooo}(h)(x_0)$ or $g^{-1}h \in C$ and $g<_{\ooo} h$. It follows that $[L(\ooo'|_{\pi_1(T_j)})]=[\alpha_j]$. Now it remains to verify that the left-ordering $\mathfrak{o}'$ order-detects the tuple $(\emptyset, \emptyset; [\beta_1], \ldots, [\beta_n])$ with slopes $[\beta_i] = [\alpha_i]$ for all $i \in K$. To do this, take an arbitrary $i \in K$ and consider three cases.

    \noindent \textbf{Case 1.} $\pi_1(T_i) \subset C$. In this case $\mathfrak{o}'|_{\pi_1(T_i)} = \mathfrak{o}|_{\pi_1(T_i)}$, so $\mathcal{L}( r_i(\mathfrak{o}')) = [\alpha_i]$.

    \noindent \textbf{Cases 2 and 3.} $\pi_1(T_i) \cap C \cong \mathbb{Z}$, or $\pi_1(T_i) \cap C = \{ 1\}$. In either case, it suffices to show that $(\pi_1(T_i) \cap P(\mathfrak{o}')) \setminus C= Q \setminus C$, where $Q$ is a positive cone in $\pi_1(T_i)$ order-detecting $[\alpha_i]$.

    Since the space of left-orderings $\LO(M)$ is compact, we can find a convergent subsequence $\{\gamma^{n_l}\cdot \ooo\}_{l \in \mathbb{N}}$ of $\{\gamma^k\cdot \ooo\}_{k \in \mathbb{N}}$. Now by \cite[Lemma 3.6]{BC23} we have
    \[ \left(\lim_{l \to \infty} P(\gamma^{n_l} \cdot \mathfrak{o}) \right) \setminus C = P(\mathfrak{o}')\setminus C.
    \]
    Next, note that
    \[\left(\lim_{l \to \infty} P(\gamma^{n_l} \cdot \mathfrak{o}) \right) \setminus C = \lim_{l \to \infty} \left( P(\gamma^{n_l} \cdot \mathfrak{o}) \setminus C \right),
    \]
    %where we interpret the limit of the sequence of subsets $P(\gamma^{n_l} \cdot \mathfrak{o}) \setminus C$, $l \geq 0$, as the limit of a sequence in $\{0,1\}^{\pi_1(M)}$. It follows that $$P(\ooo')\setminus C= (\lim_{l \to \infty} P(\gamma^{n_l}))\setminus C.$$ 
    and since the restriction map $r_i: \{0,1\}^{\pi_1(M)} \rightarrow \{0,1\}^{\pi_1(T_i)}, P \mapsto P \cap \pi_1(T_i)$ is continuous for all $i$, the limit above gives
    \[ r_i(P(\mathfrak{o}')\setminus C) = \lim_{l \to \infty} r_i(P(\gamma^{n_l} \cdot \mathfrak{o})\setminus C).\] And therefore
    \[ (\pi_1(T_i) \cap P(\mathfrak{o}'))\setminus C = (\lim_{l \to \infty} \pi_1(T_i) \cap P(\gamma^{n_l} \cdot \mathfrak{o}) )\setminus C.\]

    Since $\mathcal{L}(r_i(\gamma^{n_l} \cdot \mathfrak{o}))=[\alpha_i]$ for all $n_l \geq 0$, and there are precisely four (resp. two) positive cones in $\pi_1(T_i)$ corresponding to left-orderings detecting $[\alpha_i]$ when $[\alpha_i]$ is rational (resp. irrational), we can choose a subsequence $\{n_k\}_{k \in \mathbb{N}}$ of $\{n_l\}_{l \in \mathbb{N}}$ such that $P(\gamma^{n_k} \cdot \mathfrak{o}) \cap \pi_1(T_i) = Q$ is constant for all $k$.

    Then the limit becomes
    \[ (\pi_1(T_i) \cap P(\mathfrak{o}'))\setminus C = Q \setminus C,
    \]
    showing that $\ooo'$ order-detects $(\emptyset, \emptyset; [\beta_1], \ldots, [\beta_n])$, where $[\beta_i] = [\alpha_i]$ for all $i \in K$.
\end{proof}

\begin{theorem}
    \label{thm:generalcofinal}
    Suppose that $M$ is a compact, connected, orientable, irreducible $3$-manifold whose boundary consists of incompressible tori $T_1, \dots ,T_n$.
    % \begin{enumerate}
    % \item 
    Suppose that $\ooo$ is a left-ordering of $\pi_1(M)$ order-detecting $(\emptyset, \emptyset; [\alpha_1], \dots, [\alpha_n])$ and that $C$ is an $\ooo$-convex subgroup of $\pi_1(M)$. Let $j\in\{1,\dots,n\}$ be fixed
    %\todo{Again, we can always assume $j=1$.} 
    and $I\subset \{1,\dots,n\}$. If $\pi_1(T_j) \subset C$ and $\pi_1(T_i) \not \subset C$ for all $i \in I$, then for all $\beta \in \mathcal{S}(T_j)$ there exists a left-ordering $\ooo'$ of $\pi_1(M)$ that order-detects $(\emptyset, \emptyset; [\beta_1], \ldots, [\beta_n])$ where $[\beta_j] = [\beta]$ and $[\beta_i] = [\alpha_i]$ for all $i \in I$.

    % \item Suppose that $\ooo$ is a left-ordering of $\pi_1(M)$ detecting $(\emptyset, K; [\alpha_1], \dots, [\alpha_n])$ and that $C$ is a normal $\ooo$-convex subgroup of $\pi_1(M)$. Given $j\in\{1,\dots,n\}\setminus K$, if $\pi_1(T_j) \subset C$ and $\pi_1(T_i) \not \subset C$ for all $i \in K$, then for all $\beta \in \mathcal{S}(T_j)$ there exists a left-ordering $\ooo'$ of $\pi_1(M)$ that detects $(\emptyset, K; [\beta_1], \ldots, [\beta_n])$ where $[\beta_j] = [\beta]$ and $[\beta_i] = [\alpha_i]$ for all $i \in K$.
    %$i \in K$, then for all $\beta \in \mathcal{S}(T_j)$ there exists a left-ordering $\ooo'$ of $\pi_1(M)$ detecting $(\emptyset, K; [\beta_1], \ldots, [\beta_n])$ where $[\beta_i] = [\alpha_i]$ for all $i \in K$ and $[\beta_j] = [\beta]$.
\end{theorem}
\begin{proof}
    The proof is a slight modification of the proof of \cite[Proposition 5.3]{BC23}. We sketch the proof and its modifications here, but do not repeat all details.

    Firstly, note that $C$ is of infinite index because it is $\ooo$-convex. Let $W\to M$ be a covering such that $\pi_1(W)=C$. Then $W$ is non-compact, and moreover, $T_j\subset \partial M$ lifts to a torus $T\subset \partial W$ since $\pi_1(T_j)\subset C$. Set \[Z=\{[\gamma]\in \mathcal{S}_{rat}(T)\mid \mbox{the inclusion-induced map }H_1(T)\to H_1(W(\gamma)) \mbox{ is zero}\},\] where $W(\gamma)$ is the manifold obtained by the Dehn filling $T$ with slope $\gamma$. Let $Z^*$ be the union of $Z$ and the set of rational slopes $[\gamma]\in \mathcal{S}(T)$ such that $W(\gamma)$ is reducible. Then the series of claims made in the proof of \cite[Proposition 5.3]{BC23} shows that $Z^*$ is a nowhere dense subset of $\mathcal{S}(T)$ and for each $[\gamma]\in \mathcal{S}_{rat}(T)\setminus Z^*$, $\pi_1(W(\gamma))$ is left-orderable.

    For each $[\gamma]\in \mathcal{S}_{rat}(T)\setminus Z^*$, note that $C=\pi_1(W)$ and $C/\inner{\inner{\gamma}}_C=\pi_1(W(\gamma))$. The short exact sequence \[\xymatrix{\{1\}\ar[r]&\inner{\inner{\gamma}}_C\ar[r]& C\ar[r]&C/\inner{\inner{\gamma}}_C\ar[r]&\{1\}}\] gives rise to a lexicographic left-ordering $\ooo_C$ on $C$ for which $\inner{\inner{\gamma}}_C$ is a proper $\ooo_C$-convex subgroup. Since $\inner{\inner{C}}_C\cap \pi_1(T)=\inner{\gamma}$ (see \cite[Proof of Proposition 5.3, Claim 3]{BC23}), $\inner{\gamma}$ is a proper $\ooo_C|_{\pi_1(T_j)}$-convex subgroup of $\pi_1(T_j)\leq C$. It follows from \cite[Proposition 5.2]{BC23} that $P(\ooo_C)\sqcup (P(\ooo)\setminus C)$ is a positive cone of a left-ordering $\ooo_\gamma$ on $\pi_1(M)$. Now it is clear that $[L(\ooo_\gamma|_{\pi_1(T_j)})]=[L(\ooo_C|_{\pi_1(T_j)})]=[\gamma]$ by their constructions. Also note that for $i\in I$, since $\pi_1(T_i) \not \subset C$ and $P(\ooo_\gamma)\setminus C=P(\ooo)\setminus C$, we have $[L(\ooo_\gamma|_{\pi_1(T_i)})] = [L(\ooo|_{\pi_1(T_i)})]$. Hence $\ooo_\gamma$ order-detects $(\emptyset, \emptyset; [\beta_1], \ldots, [\beta_n])$ where $[\beta_j] = [\gamma]$ and $[\beta_i] = [\alpha_i]$ for all $i \in I$. This shows that the conclusion of this theorem holds for all $[\beta] \in \mathcal{S}_{rat}(T_j)\setminus Z^*$.

    It remains to show this theorem for $[\beta]\in \mathcal{S}(T_j)$ with $[\beta] \notin \mathcal{S}_{rat}(M) \setminus Z^*$. Since $\mathcal{S}_{rat}(M) \setminus Z^*$ is dense in $\mathcal{S}(M)$, we can pick a sequence $\{[\gamma_l]\}\subset \mathcal{S}_{rat}(M) \setminus Z^*$ that converges to $[\beta]$. Moreover, since $\LO(M)$ is compact, the sequence $\{\ooo_{\gamma_l}\}$, where each left-ordering $\ooo_{\gamma_l}$ is constructed as in the last paragraph, admits a convergent subsequence $\{\ooo_{\gamma_{l_t}}\}$, say converging to $\ooo_{\beta}$. Since $\ooo_{\gamma_{l_t}}$ order-detects $(\emptyset, \emptyset; [\beta_1], \ldots, [\beta_n])$ where $[\beta_j] = [\gamma_{l_t}]$ and $[\beta_i] = [\alpha_i]$ for all $i \in I$, it follows
    %\todo{I moved Lemma up so that we do not need to explain one more time why ``it follows''. } 
    that $\ooo_{\beta}$ order-detects $(\emptyset, \emptyset; [\beta_1], \ldots, [\beta_n])$ where $[\beta_j] = [\beta]$ and $[\beta_i] = [\alpha_i]$ for all $i \in I$.

    %Now to show (2). The proof proceeds identically up until we have created the left-ordering $\ooo_{\gamma}$ of $\pi_1(M)$. Then we note that for $i \in K$, since $g\pi_1(T_i)g^{-1} \not \subset C$ for all $g \in \pi_1(M)$, we have $$[L((g \cdot \ooo_\gamma)|_{\pi_1(T_i)})] = [L(\ooo_\gamma|_{g\pi_1(T_i)g^{-1}})] = [L(\ooo|_{g\pi_1(T_i)g^{-1}})] = [L((g \cdot \ooo)|_{\pi_1(T_i)})] = [\alpha_i],$$
    % since $P(\ooo_\gamma)\setminus C=P(\ooo)\setminus C$ (here we use \cite[Corollary 7.4]{BC23}, which shows that $[L(\ooo|_{g\pi_1(T_i)g^{-1}})] = [L((g \cdot \ooo)|_{\pi_1(T_i)})]$). Thus $\ooo_\gamma$ detects $(\emptyset, K; [\beta_1], \ldots, [\beta_n])$ where $[\beta_j] = [\gamma]$ and $[\beta_i] = [\alpha_i]$ for all $i \in K$. Thus conclusion (2) holds for all $[\beta] \in \mathcal{S}_{rat}(T_j)\setminus Z^*$.

    % Now for $\beta$ in the complement of $\mathcal{S}_{rat}(T_j)\setminus Z^*$, we apply the same convergence argument, noting that the sequence $g \cdot \ooo_{\gamma_{l_t}}$ converges to $g \cdot \ooo_{\beta}$ since $g$ acts on $\mathrm{LO}(M)$ by homeomorphisms. Thus the argument yields a left-ordering $\ooo_{\beta}$ that detects $(\emptyset, K; [\beta_1], \ldots, [\beta_n])$ where $[\beta_j] = [\beta]$ and $[\beta_i] = [\alpha_i]$ for all $i \in K$.
\end{proof}

The following results will be used to `enlarge' intervals of non-detected slopes in the coming sections.

\begin{theorem}
    \label{thm:2boundarycase}
    Suppose that $L=L_1\cup L_2$ is a hyperbolic two-component link in $S^3$. Denote the link complement by $M$ and suppose that no proper, relatively convex subgroup of $\pi_1(M)$ contains both $\pi_1(T_1)$ and $\pi_1(T_2)$. If $(\emptyset, \{1\}; [\alpha_1], [\alpha_2])$ is order-detected by $\ooo$ and $\pi_1(T_2)$ is $\ooo$-bounded, then $(\emptyset, \emptyset; [\alpha_1], [\beta])$ is order-detected for all $\beta \in \mathcal{S}(T_2)$.
\end{theorem}

\begin{proof}
    By Lemma \ref{lem:Cexists}, there is a proper subgroup $C$ and a left-ordering $\ooo'$ of $\pi_1(M)$ such that $C$ is $\ooo'$-convex and $\pi_1(T_2) \subset C$, and $\ooo'$ order-detects $(\emptyset, \emptyset; [\alpha_1], [\beta'])$ for some $[\beta'] \in \mathcal{S}(T_2)$. Note that $\pi_1(T_1) \not\subset C$ by assumption, so we can apply Theorem \ref{thm:generalcofinal} to $\ooo'$, and conclude that for every $\beta \in \mathcal{S}(T_2)$ there exists a left-ordering $\ooo''$ order-detecting $(\emptyset, \emptyset; [\alpha_1], [\beta])$.
\end{proof}

\begin{theorem}\label{thm:loimplieswkdet}
    Suppose that $L=L_1\cup L_2$ is a hyperbolic two-component link in $S^3$ with $M$ being the link complement. For $i=1,2$, let $a_i/b_i$ be rational numbers in lowest terms, and let $\{m_i,l_i\}$ be the peripheral system consisting of a meridian and longitude along $T_i$. Further assume that $\pi_1(T_1) \not\subset \inner{\inner{m_1^{a_1}l_1^{b_1}}}$ and $\pi_1(T_2) \not \subset \inner{\inner{m_2^{a_2}l_2^{b_2}}}$. If $\pi_1(S^3_{a_1/b_1,a_2/b_2}(L))$ is left-orderable, then one of $(\{1 \}, \{1,2\}; [m_1^{a_1}l_1^{b_1}], [m_2^{a_2}l_2^{b_2}])$ or $(\{2 \}, \{1,2\}; [m_1^{a_1}l_1^{b_1}], [m_2^{a_2}l_2^{b_2}])$ is order-detected, and in general, $(\emptyset, \{1,2\}; [m_1^{a_1}l_1^{b_1}], [m_2^{a_2}l_2^{b_2}])$ is always order-detected.
\end{theorem}
\begin{proof}
    Set $C = \inner{\inner{m_1^{a_1}l_1^{b_1}, m_2^{a_2}l_2^{b_2}}}$. Note that $\pi_1(S^3_{a_1/b_1,a_2/b_2}(L))$ is nontrivial since it is left-orderable, which means that we cannot have $\pi_1(T_1), \pi_1(T_2) \subset C$ since $\pi_1(M)$ is generated by the conjugates of $\{m_1, m_2\}$. So there are two cases to consider.

    As a first case, suppose that $\pi_1(T_i) \not \subset C$ for $i=1, 2$. Construct a lexicographic left-ordering $\ooo$ of $\pi_1(M)$ using the short exact sequence
    \[ \{1\} \longrightarrow C \longrightarrow \pi_1(M) \longrightarrow \pi_1(S^3_{a_1/b_1,a_2/b_2}(L)) \longrightarrow \{1\},
    \]
    so that $C$ is $\ooo$-convex. Observe that $C \cap \pi_1(T_i)$ is a proper subgroup of $\pi_1(T_i)$ that contains $m_i^{a_i}l_i^{b_i}$, and the quotient $\pi_1(M)/C$ is torsion-free. This forces $C \cap \pi_1(T_i) = \langle m_i^{a_i}l_i^{b_i} \rangle$ for $i=1,2$.   Note also that because $C$ is normal, $C$ is $g \cdot \ooo$-convex for all $g \in \pi_1(M)$, and therefore $C \cap \pi_1(T_i) = \langle m_i^{a_i}l_i^{b_i} \rangle$ is $(g \cdot \ooo)|_{\pi_1(T_i)}$-convex in $\pi_1(T_i)$ for $i=1,2$.  This means that $s(g \cdot \ooo)=([m_1^{a_1}l_1^{b_1}], [m_2^{a_2}l_2^{b_2}])$ and  thus $(\{1,2\}, \{1,2\}; [m_1^{a_1}l_1^{b_1}], [m_2^{a_2}l_2^{b_2}])$ is order-detected by $\ooo$.

    % In preparation for the next case, we first argue that both $\pi_1(S^3_{a_1/b_1, *}(L))$ and $\pi_1(S^3_{*, a_2/b_2}(L))$ are left-orderable, considering only the case of $\pi_1(S^3_{a_1/b_1, *}(L))$ in detail. \todo{The arguments for left-orderability that follow need to be considered carefully, I made them up quickly...}

    % Note that if $S^3_{a_1/b_1, *}(L)$ is irreducible, then $\pi_1(S^3_{a_1/b_1, *}(L))$ is left-orderable by \cite{BRW2005}. So suppose that $S^3_{a_1/b_1, *}(L)$ is reducible, say
    % \[S^3_{a_1/b_1, *}(L) \cong N_1 \sharp \dots \sharp N_k,
    % \]
    % where the $N_i$ are either irreducible or $S^1 \times S^2$. Note that $T_2 \subset N_i$ for some $i$, and so
    % \[S^3_{a_1/b_1, a_2/b_2}(L) \cong N_1 \sharp \dots \sharp N_i(a_2/b_2) \sharp \dots \sharp N_k.
    % \]
    % Therefore
    % \[ \pi_1(N_1) * \dots * \pi_1(N_i(a_2/b_2))* \dots * \pi_1(N_k) = \pi_1(S^3_{a_1/b_1,a_2/b_2}(L))
    % \]
    % is left-orderable, in particular $\pi_1(N_j)$ is left-orderable for all $i \neq j$ and $\pi_1(N_j)$ is left-orderable by \cite{BRW2005}. In particular, $\pi_1(S^3_{a_1/b_1, *}(L)) = \pi_1(N_1) * \dots * \pi_1(N_k)$ is left-orderable.

    On the other hand, suppose that one of $\pi_1(T_1), \pi_1(T_2)$ is contained in $C$ and the other is not, say $\pi_1(T_1) \subset C$ and $\pi_1(T_2) \not \subset C$. Construct a lexicographic left-ordering $\ooo'$ of $C$ using the short exact sequence
    \[ \{1\}\longrightarrow \inner{\inner{m_1^{a_1}l_1^{b_1}}} \longrightarrow C \longrightarrow C/\inner{\inner{m_1^{a_1}l_1^{b_1}}} \longrightarrow \{1\}.
    \]
    Note that $\pi_1(T_1) \not \subset \inner{\inner{m_1^{a_1}l_1^{b_1}}}$
    %\todo{Theoretically speaking, it is $\inner{\inner{m_1^{a_1}l_1^{b_1}}}_C$ here.  Ans:  No, I think here we want the normal closure in $\pi_1(M)$} is nontrivial because $\pi_1(T_1) \not \subset \inner{\inner{m_1^{a_1}l_1^{b_1}}}$ 
    implies $C \neq \inner{\inner{m_1^{a_1}l_1^{b_1}}}$.  Moreover $C/\inner{\inner{m_1^{a_1}l_1^{b_1}}}$ is left-orderable because it is a subgroup of $\pi_1(S^3_{a_1/b_1, *}(L))$, which is left-orderable since $S^3_{a_1/b_1, *}(L)$ is an irreducible manifold with infinite first homology \cite{BRW2005}. Next, consider the short exact sequence
    \[\{1\} \longrightarrow C \longrightarrow \pi_1(M) \longrightarrow \pi_1(S^3_{a_1/b_1,a_2/b_2}(L)) \longrightarrow \{1\}
    \]
    and construct a lexicographic left-ordering $\ooo''$ of $\pi_1(M)$ using $\ooo'$ on the subgroup $C$. By our construction, both $C$ and $\inner{\inner{m_1^{a_1}l_1^{b_1}}}$ are normal and $\ooo''$-convex.  Arguing as above, $\inner{\inner{m_1^{a_1}l_1^{b_1}}}$ is $g \cdot \ooo''$-convex for all $g \in \pi_1(M)$, and so $\inner{\inner{m_1^{a_1}l_1^{b_1}}} \cap \pi_1(T_1) = \langle m_1^{a_1}l_1^{b_1} \rangle$ is $(g \cdot \ooo'')|_{\pi_1(T_1)}$-convex in $\pi_1(T_1)$ (here we use that $\pi_1(T_1) \not \subset \inner{\inner{m_1^{a_1}l_1^{b_1}}}$).  We can similarly analyse the restriction of $g \cdot \ooo''$ to $\pi_1(T_2)$ and conclude that $(\{2 \}, \{1,2\}; [m_1^{a_1}l_1^{b_1}], [m_2^{a_2}l_2^{b_2}])$ is order-detected by $\ooo''$.

    In the case where $\pi_1(T_2) \subset C$ and $\pi_1(T_1) \not \subset C$, we proceed similarly and deduce that $(\{1 \}, \{1,2\}; [m_1^{a_1}l_1^{b_1}], [m_2^{a_2}l_2^{b_2}])$ is $\ooo''$-detected (here we use that $\pi_1(T_2) \not \subset \inner{\inner{m_2^{a_2}l_2^{b_2}}}$).

    Finally, we remark that regardless of cases, $(\emptyset, \{1,2\}; [m_1^{a_1}l_1^{b_1}], [m_2^{a_2}l_2^{b_2}])$ is always order-detected by Remark \ref{rem:sub-detection}.
\end{proof}

\section{Applications to knots and two-component links}
\label{sec:families}
\subsection{An infinite family of links}  %\todo{I have made it this far in my editing}

Let $\LLL'=L_0\cup L_1\cup L_2$ be the three-component link as shown in Figure \ref{fig:linklpp}. For $n\in\NN$, denote by $\LLL_n = L_1\cup L_2$ the link as shown in Figure \ref{fig:linklnn}, where the first component $L_1$ is the torus knot $T(2,2n+1)$ and $L_2$ is the unknot. In the Thistlethwaite Link Table and up to mirror images, $\LLL_0$ is $L5a1$ (the Whitehead link); $\LLL_1$ is $L7a3$; $\LLL_2$ is $L9a14$; and $\LLL_3$ is $L11a110$.
\begin{figure}[h]
    \begin{center}
        \includegraphics[width = 4.5cm]{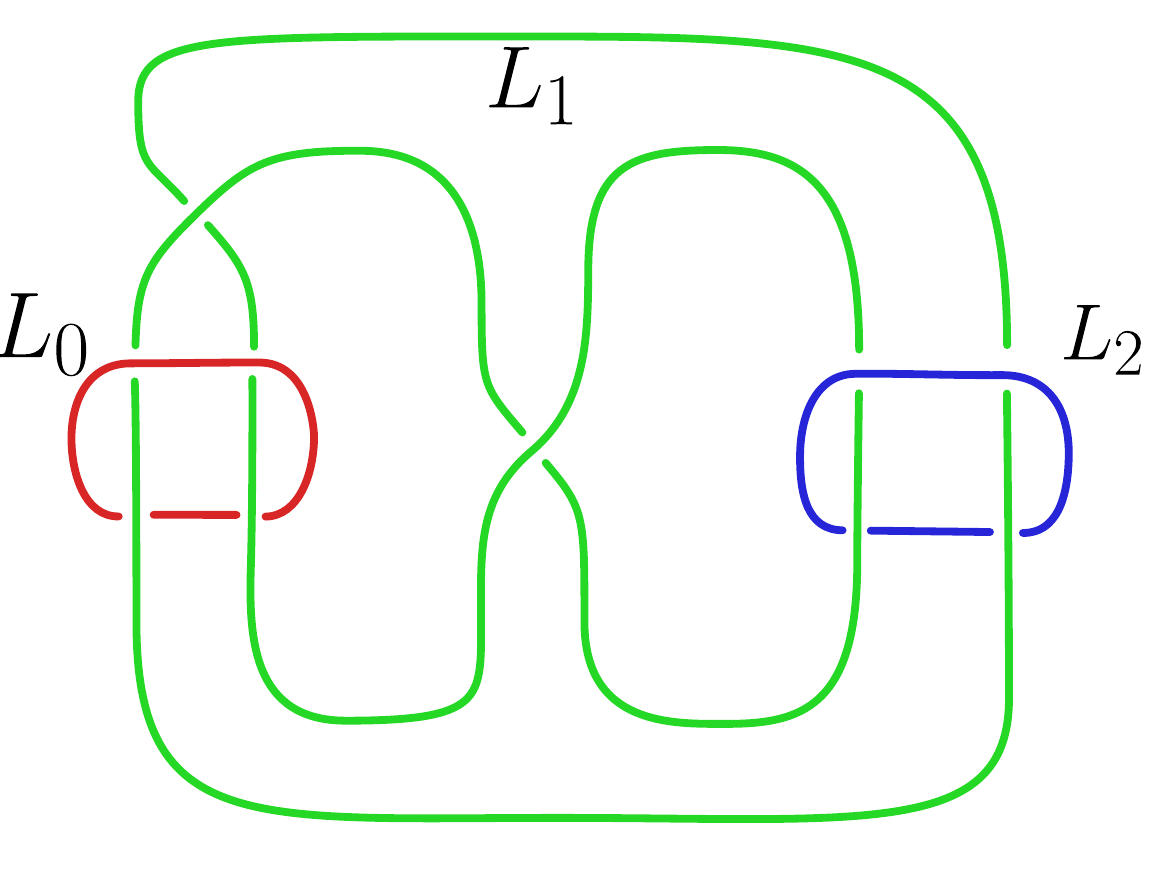}
    \end{center}
    \caption{The three-component link $\LLL'$}
    \label{fig:linklpp}
\end{figure}

The fundamental group of the link complement of $\LLL'$ is given by SnapPy \cite{SnapPy} as \[\Gamma=\inner{x,y,z\mid xz^2y^2x^{-1}z^{-1}y^{-2}z^{-1},\, x^2zy^{-1}z^{-2}x^{-2}yz},\] together with the peripheral systems of $L_0, L_1$ and $L_2$ respectively as \begin{align*}
    m_0 & =z^{-1}x^{-2}y, &  & l_0=zx^2z;                 \\
    m_1 & =z^{-1}x^{-1},  &  & l_1=xy^{-2}z^{-1}(xz)^3;   \\
    m_2 & =xzy,           &  & l_2=yzxy^{-1}z^{-1}x^{-1}.
\end{align*}

From here forward, we denote by $M_n$ the link complement of the link $\LLL_n$. Since $\LLL_n$ can be obtained by a Rolfsen twist along the first component $L_0$ of $\LLL'$, the fundamental group $\pi_1(M_n)$ has the following presentation: \[\pi_1(M_n) = \inner{x,y,z\mid xz^2y^2x^{-1}z^{-1}y^{-2}z^{-1},\, x^2zy^{-1}z^{-2}x^{-2}yz,\, m_0^{-1}l_0^n}.\]
Write $r_1=xz^2y^2x^{-1}z^{-1}y^{-2}z^{-1}$ and $r_2=x^2zy^{-1}z^{-2}x^{-2}yz$. From $m_0^{-1}l_0^n=1$, we see that $y=(x^2z^2)^nx^2z$. Substituting $y=(x^2z^2)^nx^2z$ into $r_2$, we see that $r_2$ is killed automatically. Substituting $y=(x^2z^2)^nx^2z$ into $r_1$, we obtain the following presentation: \[\pi_1(M_n) = \inner{x,z\mid x(z^2x^2)^{n+1}z(x^2z^2)^{n+1}z^{-1}x^{-1}(z^{-2}x^{-2})^{n+1}z^{-1}(x^{-2}z^{-2})^{n+1}z}.\]

Replacing $x$ with $a$ and $z$ with $a^{-2}b$, we obtain the presentation:
\[\pi_1(M_n) = \inner{a,b\mid (a^2b^{-2})^na^2b^{-1}a^{-1}(b^2a^{-2})^nb^3 = b^3(a^{-2}b^2)^na^{-1}b^{-1}a^2(b^{-2}a^2)^n}.\] Under this replacement, the peripheral elements become: \begin{align*}
    m_1 & = b^{-1}a,                & l_1 & = a(b^{-2}a^2)^nb^{-3}(a^2b^{-2})^na^2(a^{-1}b)^3; \\
    m_2 & = a^{-1}b^2(a^{-2}b^2)^n, & l_2 & = b(a^{-2}b^2)^na^{-2}ba(b^{-2}a^2)^nb^{-2}a.
\end{align*} However, our new peripheral systems must take into account the change of framing when performing the Rolfsen twist along $L_0$ of $\LLL'$ (see e.g. \cite[\S 16]{prasolov1997knots}). It follows that the peripheral systems along $L_1$ and $L_2$ in $M_n$ are given respectively by \begin{align*}
    m   & = b^{-1}a,                & l       & = a(b^{-2}a^2)^nb^{-3}(a^2b^{-2})^na^2(a^{-1}b)^{4n+3}; \\
    \mu & = a^{-1}b^2(a^{-2}b^2)^n, & \lambda & = b(a^{-2}b^2)^na^{-2}ba(b^{-2}a^2)^nb^{-2}a,
\end{align*}
where $\{m, l\}$ and $\{\mu, \lambda\}$ serve as generators for the subgroups $\pi_1(T_1)$ and $\pi_1(T_2)$ respectively. Note that for clarity in the arguments that follow, we have renamed the meridian and longitude pairs to avoid subscripts. We note the following formulas, to be used in the computations below: \begin{align*}
    m^{4n+2}l & = a(b^{-2}a^2)^nb^{-3}(a^2b^{-2})^nab,  & \mu\lambda   & = b(a^{-2}b^2)^na^{-2}ba = ba^{-1}\mu m, \\
    m^{4n+3}l & = a(b^{-2}a^2)^nb^{-3}(a^2b^{-2})^na^2, & \mu^2\lambda & = ba^{-2}(b^2a^{-2})^nb^3(a^{-2}b^2)^n.
\end{align*}

\begin{theorem}\label{thm:lnnnotweak}
    If $([\alpha_1],[\alpha_2])\in (4n+2,\infty)\times (2,\infty)\subset \mathcal{S}(M_n)$, then $(\emptyset,\emptyset;[\alpha_1],[\alpha_2])$ is not order-detected.
\end{theorem}
\begin{proof}
    Assume that $([\alpha_1],[\alpha_2])\in (4n+2,\infty)\times (2,\infty)\subset \mathcal{S}(M_n)$, and that $(\emptyset,\emptyset;[\alpha_1],[\alpha_2])$ is order-detected by some left-ordering $\ooo$ of $\pi_1(M_n)$. Then for all sufficiently large integers $N$, $m^{4n+2}l$ and $m^Nl$ are of opposite signs and $\mu^2\lambda$ and $\mu^N\lambda$ are also of opposite signs under $\ooo$. Replacing $\ooo$ with its opposite if necessary, we can further assume $m^{4n+2}l<_\ooo 1<_\ooo m^Nl$ and therefore $m^{4n+2}l<_\ooo 1<_\ooo m$. It follows that $a$ and $b$ are of the same sign, for otherwise $m^{4n+2}l$ and $m$ would have the same sign. We consider cases based on the signs of $a,b$ and $\mu^2\lambda$, and show there is a contradiction in each case.

    \begin{casesp}
        \item $\mu^2\lambda>_\ooo 1$.

        If $\mu>_\ooo 1$, then it follows immediately that $\mu^{t+2}\lambda>_\ooo 1$ for all $t\in \NN$, which is a contradiction. So $\mu<_\ooo 1$ in this case.
        \begin{casesp}
            \item $a,b$ are positive.

            Observe that one of $a^{-1}b^2 >_\ooo 1$ or $ba^{-1} >_\ooo 1$ must hold; for if they are both negative, then the expression
            \begin{align*}
                \mu^2 \lambda & = ba^{-2}(b^2a^{-2})^nb^3(a^{-2}b^2)^n                                                   \\
                              & = (ba^{-1})[(a^{-1}b^2)a^{-1}]^n(a^{-1}b^2)(ba^{-1})[(a^{-1}b^2)a^{-1}]^{n-1}(a^{-1}b^2)
            \end{align*}
            would imply that $\mu^2 \lambda <_\ooo 1$, a contradiction.

            Suppose that $ba^{-1} >_\ooo 1$. We see that \[\mu = a^{-1}b^2(a^{-2}b^2)^n = a^{-1}b^{-1}(a^2b^{-2})^nab(m^{4n+2}l)^{-1}a.\] Hence, for all $2< t\in \NN$, we have \begin{align*}
                (\mu\lambda)\mu^t & = b(a^{-2}b^2)^na^{-2}ba(a^{-1}b^{-1}(a^2b^{-2})^nab(m^{4n+2}l)^{-1}a)^t                             \\
                                  & = b(a^{-2}b^2)^na^{-2}(a^2b^{-2})^nab((m^{4n+2}l)^{-1}b^{-1}(a^2b^{-2})^nab)^{t-1}(m^{4n+2}l)^{-1} a \\
                                  & = ba^{-1}b((m^{4n+2}l)^{-1}b^{-1}(a^2b^{-2})^nab)^{t-1}(m^{4n+2}l)^{-1} a.
            \end{align*}
            But now $b^{-1}(a^2b^{-2})^nab=b((b^{-2}a^2)^nb^{-2}a)b=b\mu^{-1}b$ is positive, as are $ba^{-1}$, $a,b$ and $(m^{4n+2}l)^{-1}$. This implies $\mu^{t+1}\lambda$ is positive for all $2< t\in\NN$, which is a contradiction.

            On the other hand, suppose $a^{-1}b^2 >_\ooo 1$. Making use of the relator, we see that \[\mu = a^{-1}b^2(a^{-2}b^2)^n = a^{-1}b^{-1}(a^2b^{-2})^na^2(m^{4n+2}l)^{-1}a^{-1}ba.\]
            Then for $2 < t \in \NN$ we can write
            \begin{align*}
                (\mu\lambda)\mu^t & = b(a^{-2}b^2)^na^{-2}ba(a^{-1}b^{-1}(a^2b^{-2})^na^2(m^{4n+2}l)^{-1}a^{-1}ba)^t                              \\
                                  & = b(a^{-2}b^2)^na^{-2}(a^2b^{-2})^na^2((m^{4n+2}l)^{-1}a^{-1}(a^2b^{-2})^na^2)^{t-1}(m^{4n+2}l)^{-1} a^{-1}ba \\
                                  & = b((m^{4n+2}l)^{-1}a^{-1}(a^2b^{-2})^na^2)^{t-1}(m^{4n+2}l)^{-1} a^{-1}ba                                    \\
                                  & =b((m^{4n+2}l)^{-1}(a^{-1}b^2)(b^{-2}a^2)^nb^{-2}a^2)^{t-1}(m^{4n+2}l)^{-1} (a^{-1}b^2)(b^{-1}a).
            \end{align*}
            But then $(b^{-2}a^2)^nb^{-2}a^2 = \mu^{-1}a$ is positive, as are $a^{-1}b^2$, $a,b$, $b^{-1}a=m$, and $(m^{4n+2}l)^{-1}$, which implies $\mu^{t+1}\lambda$ is positive for all $2< t\in\NN$ in this case. So we have reached a contradiction in this case.

            % \item $a,b$ are negative.

            % Note that \[m = b^{-1}a = a(\mu\lambda)^{-1}b(a^{-2}b^2)^na^{-1}.\]
            % For $2< t\in \NN$, we have \begin{align*}
            % m^t(m^{4n+3}l) & = (a(\mu\lambda)^{-1}b(a^{-2}b^2)^na^{-1})^t (a(b^{-2}a^2)^nb^{-3}(a^2b^{-2})^na^2) \\
            % & = a((\mu\lambda)^{-1}b(a^{-2}b^2)^n)^{t-1}(\mu\lambda)^{-1}b(a^{-2}b^2)^na^{-1}(a(b^{-2}a^2)^nb^{-3}(a^2b^{-2})^na^2) \\
            % & = a((\mu\lambda)^{-1}b(a^{-2}b^2)^n)^{t-1}(\mu\lambda)^{-1}b^{-2}(a^2b^{-2})^na^2 \\
            % & = a((\mu\lambda)^{-1}b(a^{-2}b^2)^n)^{t-1}(\mu\lambda)^{-1}\mu^{-1}a \\
            % & = a((\mu\lambda)^{-1}b(a^{-2}b^2)^n)^{t-1}(\mu^2\lambda)^{-1}a.
            % \end{align*}
            % Since $b(a^{-2}b^2)^n = a\mu<_\ooo 1$\todo{I think this should be $m \mu$ instead of $a\mu$? If yes then there is a problem here, likely easily fixed.} and $a, (\mu\lambda)^{-1}, (\mu^2\lambda)^{-1}$ are also negative, $m^t(m^{4n+3}l)$ is also negative. But this gives us a contradiction.

            \item $a,b$ are negative.

            First suppose that $a^{-2}b^2 <_\ooo 1$. Note that \[m = b^{-1}a = a(\mu\lambda)^{-1}b(a^{-2}b^2)^na^{-1}.\]
            For $2< t\in \NN$, we have \begin{align*}
                m^t(m^{4n+2}l) & = (a(\mu\lambda)^{-1}b(a^{-2}b^2)^na^{-1})^t (a(b^{-2}a^2)^nb^{-3}(a^2b^{-2})^nab)                                   \\
                               & = a((\mu\lambda)^{-1}b(a^{-2}b^2)^n)^{t-1}(\mu\lambda)^{-1}b(a^{-2}b^2)^na^{-1}(a(b^{-2}a^2)^nb^{-3}(a^2b^{-2})^nab) \\
                               & = a((\mu\lambda)^{-1}b(a^{-2}b^2)^n)^{t-1}(\mu\lambda)^{-1}b^{-2}(a^2b^{-2})^nab                                     \\
                               & = a((\mu\lambda)^{-1}b(a^{-2}b^2)^n)^{t-1}(\mu\lambda)^{-1}\mu^{-1}b                                                 \\
                               & = a((\mu\lambda)^{-1}b(a^{-2}b^2)^n)^{t-1}(\mu^2\lambda)^{-1}b.
            \end{align*}
            Since $(a^{-2}b^2)^n <_\ooo 1$ and $a, b, (\mu\lambda)^{-1}, (\mu^2\lambda)^{-1}$ are also negative, $m^t(m^{4n+3}l)$ is negative. But this gives us a contradiction.

            Next suppose that $b^{-2}a^2 <_\ooo 1$. Note that
            \[ m = b^{-1}a = b((b^{-2}a^2)(a^{-1}b))b^{-1}
            \]
            and therefore if $t \geq2$ then
            \begin{align*}
                m^t & = (b((b^{-2}a^2)(a^{-1}b))b^{-1})^t                         \\
                    & = b((b^{-2}a^2)(a^{-1}b))^{t-1}((b^{-2}a^2)(a^{-1}b))b^{-1} \\
                    & = b((b^{-2}a^2)(a^{-1}b))^{t-1}(b^{-2}a^2)a^{-1}.
            \end{align*}
            Now, since $m^{4n+3}l = a(\mu^2 \lambda)^{-1}b$, we can write
            \[m^t m^{4n+3}l = b((b^{-2}a^2)(a^{-1}b))^{t-1}(b^{-2}a^2)(\mu^2 \lambda)^{-1}b.\]
            Since $b, b^{-2}a^2, a^{-1}b$ and $(\mu^2 \lambda)^{-1}$ are negative, $m^t m^{4n+3}l$ is also negative for all $t \geq 2$, a contradiction.

        \end{casesp}
        \item $\mu^2\lambda<_\ooo 1$.

        If $\mu<_\ooo 1$, it follows that $\mu^{t+2}\lambda<_\ooo 1$ for all $t\in \NN$, which is a contradiction. So $\mu>_\ooo 1$ in this case. Since $\mu\lambda = ba^{-1}\mu m <_\ooo 1$ and both $\mu$ and $m$ are positive, $ba^{-1}<_\ooo 1$. Also observe that \begin{align*}
            m\mu m & = (b^{-1}a) (a^{-1}b^2(a^{-2}b^2)^n)(b^{-1}a) \\
                   & = b(a^{-2}b^2)^nb^{-1}a                       \\
                   & = ((ba^{-1})(a^{-1}b))^{n}a.
        \end{align*}
        Since $ba^{-1}<_\ooo 1,a^{-1}b = m^{-1}<_\ooo 1$ and $m\mu m>_\ooo 1$, we must have $a>_\ooo 1$, and therefore $b>_\ooo 1$ as well. Since $((ba^{-1})(a^{-1}b))^{n}a = ((ba^{-1})(a^{-1}b))^{n-1}(ba^{-1})a^{-1}ba>_\ooo 1$, we also conclude that $a^{-1}ba >_\ooo 1$ and so $a^{-1}b^{-1}a <_\ooo 1$.

        % Now with $x = (b^{-2}a^2)^nb^{-3}(a^2b^{-2})^na^2$ we can write $m^{4n+2}l = a^{-1}bax$, and since $m^{4n+2}l <_\ooo 1$ this implies $x <_\ooo 1$. But now $\mu^2 \lambda = bx^{-1}$ is a product of positive elements, contradicting $\mu^2\lambda<_\ooo 1$.

        Note there is an equality $a^{-1}(m^{4n+3}l)b^{-1}(\mu^2\lambda) = 1$ that can be rewritten as \[(a^{-1}b^{-1}a)(m^{4n+2}l)b^{-1}(\mu^2\lambda) = 1.\] However $a^{-1}b^{-1}a,b^{-1},m^{4n+2}l$ and $\mu^2\lambda$ are negative, while the right-hand side is the identity. This leads to a contradiction.

    \end{casesp}
\end{proof}

% \begin{corollary}
% If $r_1>4n+3$ and $r_2>2$, then $S^3_{r_1,r_2}(\LLL_n)$ is not left-orderable.
% \end{corollary}
% \begin{proof}
% This follows from Theorem \ref{thm:lnnnotweak} and Theorem \ref{thm:loimplieswkdet}.
% \end{proof}

\begin{lemma}
    \label{lem:nopropercontainsperipherals}
    No proper subgroup of $\pi_1(M_n)$ contains both $\pi_1(T_1)$ and $\pi_1(T_2)$.
\end{lemma}
\begin{proof}
    Suppose $H \subset \pi_1(M_n)$ contains $\pi_1(T_1)$ and $\pi_1(T_2)$. Then $H$ contains $\mu^{-1}m^{-1} \mu \lambda = ba^{-1}$. From the identity $m \mu m = ((ba^{-1})(a^{-1}b))^na$ we see that $a \in H$, from which is follows easily that $b \in H$, so that $H = \pi_1(M_n)$.
\end{proof}

% Similarly, we can extend the results by examining the cofinality of the boundary subgroups.
As a result, we have the following corollary.

\begin{corollary}
    \label{cor:bddrangeLn}
    \hangindent\leftmargini
    (1)\hskip\labelsep If $[\alpha_1] \in (4n+2,\infty)$ and $(\emptyset, \{1\}; [\alpha_1], [\alpha_2])$ is order-detected by a left-ordering $\ooo$ of $\pi_1(M_n)$, then $\pi_1(T_2)$ is $\ooo$-cofinal.
    \begin{enumerate}
        \setcounter{enumi}{1}
        \item If $[\alpha_2] \in (2,\infty)$ and $(\emptyset, \{2\}; [\alpha_1], [\alpha_2])$ is order-detected by a left-ordering $\ooo$ of $\pi_1(M_n)$, then $\pi_1(T_1)$ is $\ooo$-cofinal.
    \end{enumerate}

\end{corollary}
\begin{proof}
    We prove only (1), with the argument for (2) being similar.

    Suppose that $[\alpha_1] \in (4n+2,\infty)$ and $(\emptyset, \{1\}; [\alpha_1], [\alpha_2])$ is order-detected by $\ooo$, and that $\pi_1(T_2)$ is $\ooo$-bounded. By Lemma \ref{lem:nopropercontainsperipherals} we may apply Theorem \ref{thm:2boundarycase}, concluding $(\emptyset, \emptyset; [\alpha_1], [\beta])$ is order-detected for all $\beta \in \mathcal{S}(T_2)$. But when $[\beta] \in (2, \infty)$, we know that $(\emptyset, \emptyset; [\alpha_1], [\beta])$ is not order-detected by Theorem \ref{thm:lnnnotweak}, a contradiction.
\end{proof}

\begin{lemma}
    \label{lem:cofinalintervalLn}
    %Let $M_n$ denote the complement of the link $\mathbb{L}_n$ with peripheral subgroups as above, and $\ooo$ be a left-ordering of $\pi_1(M_n)$. 
    If $[\alpha_1] \in (2n+2,\infty)$ and $(\emptyset, \emptyset; [\alpha_1], [\alpha_2])$ is order-detected by a left-ordering $\ooo$ of $\pi_1(M_n)$, then $\pi_1(T_1)$ is not $\ooo$-cofinal.
\end{lemma}
\begin{proof}
    Suppose $[\alpha_1] \in (2n+2,\infty)$ and $(\emptyset, \emptyset; [\alpha_1], [\alpha_2])$ is order-detected by $\ooo$, and that $\pi_1(T_1)$ is $\ooo$-cofinal. In particular, $m = b^{-1}a$ is $\ooo$-cofinal, we may assume that $m$ is positive and $m^{2n+2}l$ is negative.

    Using the fact (see \cite{BC23}) that the products and conjugates of positive, $\ooo$-cofinal elements are also positive and $\ooo$-cofinal, we see that $b^{-2}ab=b^{-1}mb$ and $ab^{-1}=ama^{-1}$ are positive and $\ooo$-cofinal. Also note that
    \[ \prod_{i=0}^{n-1}(a^{-1}b)^{n-i}(ab^{-1})(b^{-1}a)^{n-i} = (a^{-1}b)^n[(ab^{-1})(b^{-1}a)]^n\]
    and since the left-hand side is a product of conjugates of $m$, the right-hand side is positive and $\ooo$-cofinal. Similarly, the equality
    \[ \prod_{i=0}^{n-1}(a^{-1}b)^{n-i}(b^{-2}ab)(b^{-1}a)^{n-i} = (a^{-1}b)^n[(b^{-2}ab)(b^{-1}a)]^n\]
    allows us to conclude that the right-hand side is positive and $\ooo$-cofinal. It follows that $[(b^{-2}ab)(b^{-1}a)]^n(b^{-2}ab)(a^{-1}b)^n$ is positive and $\ooo$-cofinal as well, since it is a product of the two positive $\ooo$-cofinal terms
    $$[(b^{-2}ab)(b^{-1}a)]^n(a^{-1}b)^n, \; (a^{-1}b)^{-n}(b^{-2}ab)(a^{-1}b)^n.
    $$
    Now our final observation is that
    \[ m^{4n+2}l = [(ab^{-1})(b^{-1}a)]^n(ab^{-1})[(b^{-2}ab)(b^{-1}a)]^n(b^{-2}ab)\]
    and therefore, adding a power of $m^{-2n}$ to both sides, we get
    \[m^{2n+2}l = (a^{-1}b)^n[(ab^{-1})(b^{-1}a)]^n (ab^{-1})[(b^{-2}ab)(b^{-1}a)]^n(b^{-2}ab)(a^{-1}b)^n.\]
    But this last expression is a product of positive terms, so $m^{2n+2}l$ is positive. This contradicts the fact that $m^{2n+2}l$ is negative.
\end{proof}

Combining Lemma \ref{lem:cofinalintervalLn} and Corollary \ref{cor:bddrangeLn}, we have the following.

\begin{corollary}
    \label{cor:nondetectedLn}
    %Let $M_n$ denote the complement of the link $\mathbb{L}_n$, with peripheral subgroups as above.
    If $([\alpha_1], [\alpha_2]) \in (2n+2,\infty) \times (2, \infty)$, then $(\emptyset, \{2\}; [\alpha_1], [\alpha_2])$ is not order-detected by a left-ordering $\ooo$ of $\pi_1(M_n)$.
\end{corollary}

\begin{lemma}
    \label{lem:notcontainedLn}
    Suppose that $(a_1/b_1,a_2/b_2)\in (2n+2,\infty)\times (2,\infty)$ are rational numbers written in lowest terms with $a_i, b_i >0$, then $\pi_1(T_1) \not \subset \inner{\inner{m^{a_1}l^{b_1}}}$ and $\pi_1(T_2) \not \subset \inner{\inner{\mu^{a_2}\lambda^{b_2}}}$.
\end{lemma}
\begin{proof}
    We only argue that $\pi_1(T_1) \not \subset \inner{\inner{m^{a_1}l^{b_1}}}$, the other case being similar. Note that $H_1(M_n ; \mathbb{Z}) \cong \mathbb{Z} \oplus \mathbb{Z}$, with the copies of $\mathbb{Z}$ generated by $m$ and $\mu$. Therefore, if $\pi_1(T_1) \subset \inner{\inner{m^{a_1}l^{b_1}}}$, then the factor generated by $m$ is killed when one appropriately Dehn fills $L_1$, meaning $H_1(S^3_{a_1/b_1, a_2/b_2}(L_n); \mathbb{Z})$ is a cyclic group generated by $\mu$ whose order divides $a_2$. On the other hand, since the linking number of $L_1$ and $L_2$ is zero, by \cite[Proposition 5.3.11]{GS99} we have $ |H_1(S^3_{a_1/b_1, a_2/b_2}(L_n); \mathbb{Z})| = a_1a_2.
    $ This is a contradiction since $a_1, a_2 >1$.
\end{proof}

We are now ready to prove Theorem \ref{thm:introslopes}.
% \begin{corollary}
%     If $(r_1,r_2)\in (2n+2,\infty)\times (2,\infty)$ is a pair of rational numbers, then $S^3_{r_1,r_2}(\LLL_n)$ is not left-orderable.
% \end{corollary}
\begin{proof}[Proof of Theorem \ref{thm:introslopes}]
    If $S^3_{r_1,r_2}(\LLL_n)$ is left-orderable for some $(r_1,r_2)\in (2n+2,\infty)\times (2,\infty)$, then by Lemma \ref{lem:notcontainedLn} and Theorem \ref{thm:loimplieswkdet}, we know
    $(\emptyset, \{1,2\}; [m^{a_1}l^{b_1}], [\mu^{a_2}\lambda^{b_2}])$ is order-detected. Applying Remark \ref{rem:sub-detection}, we have a contradiction to Corollary \ref{cor:nondetectedLn}, which finishes the proof.
\end{proof}

\subsection{Whitehead link}
In this section, we denote by $\wh$ the mirror image of the Whitehead link in $S^3$ and by $M$ its link complement. Setting $n=0$ in the presentation for $\pi_1(M_n)$ of the previous section, we have
\[ \pi_1(M) = \inner{a, b \mid a^2b^{-1}a^{-1}b^3 = b^3a^{-1}b^{-1}a^2}, \]
with meridians and longitudes given by
\[m=b^{-1}a,\, l=ab^{-3}a^2(a^{-1}b)^3\,;\quad \mu=a^{-1}b^2,\, \lambda=ba^{-2}bab^{-2}a.\]

We will use the simplicity of the presentation, together with the following theorem, to improve the Dehn filling results of the previous section in the case $n=0$. Given a group $G$ and a non-identity element $g \in G$, in the discussion below we use $N(g)$ to denote the root-closed, conjugacy-closed subsemigroup of $G$ generated by $g$.

\begin{theorem}
    \label{thm:cofinaltrick}
    Suppose that $M$ is a knot manifold with peripheral subgroup $\pi_1(T)$ generated by $\{\mu, \lambda\}$, and there exist coprime integers $p, q >0$ such that $N(\mu^p \lambda^q) \cap N(\mu) \neq \emptyset$. If there exists a slope $[\alpha] \in \mathcal{S}(M)$ that is not weakly order-detected, then no $[\beta] \in (p/q, \infty)$ is weakly order-detected.
\end{theorem}
\begin{proof}
    Let $\ooo$ be a left-ordering of $\pi_1(M)$. By \cite[Theorem 1.7]{BC23}, since there exists a slope which is not weakly order-detected, $\pi_1(T)$ is $\ooo$-cofinal. Suppose that $\ooo$ order-detects the slope $[\beta] \in (p/q, \infty)$. Then in particular, $\mu^p \lambda^q$ and $\mu$ are of opposite signs, and each is cofinal in $\pi_1(T)$ and thus cofinal in $\pi_1(M)$ as well. Without loss of generality, we may assume that $\mu^p \lambda^q >_\ooo 1$ and $\mu <_\ooo 1$.

    The set of positive (resp. negative), $\ooo$-cofinal elements form a root-closed, conjugacy-closed subsemigroup of $\pi_1(M)$; see \cite{BC23}. Let us denote this subsemigroup by $N_+$ (resp. $N_-$). Then as $\mu^p \lambda^q$ is positive and cofinal, $N(\mu^p \lambda^q) \subset N_+$ and similarly $N(\mu) \subset N_-$. Yet $N_+ \cap N_- = \emptyset$, while $N(\mu^p \lambda^q) \cap N(\mu)$ is assumed to be nonempty, a contradiction.
\end{proof}

Next we confirm that certain fillings of $M$ satisfy the hypotheses of Theorem \ref{thm:cofinaltrick}.

\begin{lemma}
    \label{lem:whnormalclosure}
    If $M$ denotes the Whitehead link complement with peripheral systems as above, then $N(ml) \cap N(m) \neq \emptyset$.
\end{lemma}
\begin{proof}
    Using $ml = a^{-1}bab^{-3}ab$, one can verify that
    \begin{align*}
        (a^{-1}b^{-1}a ml a^{-1}ba)ml & = (a^{-1}b^{-1}a)(a^{-1}bab^{-3}ab)(a^{-1}ba)(a^{-1}bab^{-3}ab) \\
                                      & = b^{-3}a(ba^{-1}b^2ab^{-3}ab^{-2}a)a^{-1}b^3                   \\
                                      & = b^{-3}a(ba^{-1}b^3(b^{-1}a)b^{-3}ab^{-1}(b^{-1}a))a^{-1}b^3.
    \end{align*}
    Clearly $(a^{-1}b^{-1}a ml a^{-1}ba)ml \in N(ml)$, while the right-hand side of the equation above, being a product of conjugates of $m = b^{-1}a$, lies in $N(m)$.
\end{proof}

\begin{proposition}
    \label{prop:whprop}
    Let $p,q$ be coprime integers with $0<p \leq q$. If $[\beta] \in (3,4) \subset \mathcal{S}(S^3_{*,1+\frac{p}{q}}(\wh))$, then $[\beta]$ is not weakly order-detected.
\end{proposition}
\begin{proof}
    The special case $p=q=1$ will occasionally require a slightly different computation. Whenever necessary, we will note this exceptional case.

    Firstly, we note that \[\pi_1(S^3_{*,1+\frac{p}{q}}(\wh)) = \inner{b,a\mid a^2b^{-1}a^{-1}b^3 = b^3a^{-1}b^{-1}a^2, (\mu\lambda)^{q-p} (\mu^2\lambda)^p=1},\] where \[(\mu\lambda)^{q-p} (\mu^2\lambda)^p =(ba^{-2}ba)^{q-p}(ba^{-2}b^3)^p.\]

    Suppose $\ooo$ is a left-ordering of $\pi_1(S^3_{*,1+p/q}(\wh))$ that weakly order-detects some $[\beta]\in (3,4)\subset \mathcal{S}(S^3_{*,1+p/q}(\wh))$. Then $m^3l$ and $m^4l$ are of opposite signs under $\ooo$. Changing to the opposite of $\ooo$ if necessary, we may further assume that $m^3l=ab^{-3}a^2>_\ooo 1.$

    If $m>_\ooo 1$, then it follows immediately that $m^4l>_\ooo 1$, which is a contradiction. So we assume $m=b^{-1}a<_\ooo 1$. From $m=b^{-1}a<_\ooo 1 <_\ooo m^3l=ab^{-3}a^2$, we see that $b$ and $a$ are of the same sign. We consider cases based upon the signs of $a$ and $b$.

    \begin{casesp}
        \item Both $a,b$ are positive.

        Rewrite the relator $(\mu\lambda)^{q-p} (\mu^2\lambda)^p=1$ as \[((ba^{-2}b)a)^{q-p}((ba^{-2}b)b^2)^p=1.\] Since $b,a$ are positive, we have $ba^{-2}b<_\ooo 1$ and so $b^{-1}a^2b^{-1}>_\ooo 1$.

        Now if $a^{-2}b^2>_\ooo 1$ and $p<q$, then we can again rewrite $(\mu\lambda)^{q-p} (\mu^2\lambda)^p=1$ as \[b ((a^{-2}b^2)(b^{-1}ab))^{q-p} (a^{-2}b^2)b (b(a^{-2}b^2) b)^{p-1}=1.\] Since $b$ and $a^{-2}b^2$ are positive, we must have $b^{-1}ab<_\ooo 1$ and then $b^{-1}a^{-1}b>_\ooo 1$ and $b^{-1}a^{-2}b>_\ooo 1$. But now the relator $(\mu\lambda)^{q-p} (\mu^2\lambda)^p=1$ can be rewritten again as \[(b^2(b^{-1}a^{-2}b)a)^{q-p}(b^2(b^{-1}a^{-2}b)b^2)^p=1,\] where all the terms on the left-hand side are positive and the right-hand side is the identity. We arrive at a contradiction. Therefore, we must have $a^{-2}b^2<_\ooo 1$, or equivalently, $b^{-2}a^2>_\ooo 1$. Note that if $p=q=1$ then the relator $b(a^{-2}b^2)b = 1$ implies $b^{-2}a^2>_\ooo 1$ in this case as well.

        It follows immediately that $m^4l =(b^{-1}a^2b^{-1})(b^{-2}a^2)>_\ooo 1$, which is a contradiction.

        \item Both $a,b$ are negative.

        We shall show that this case is not possible as well. The argument begins as in the previous case: Rewrite the relator $(\mu\lambda)^{q-p} (\mu^2\lambda)^p=1$ as \[((ba^{-2}b)a)^{q-p}((ba^{-2}b)b^2)^p=1.\] Since $b,a$ are negative, we have $ba^{-2}b>_\ooo 1$ and so $b^{-1}a^2b^{-1}<_\ooo 1$.

        Now if $a^{-2}b^2<_\ooo 1$ and $p<q$, then we can again rewrite $(\mu\lambda)^{q-p} (\mu^2\lambda)^p=1$ as \[b ((a^{-2}b^2)(b^{-1}ab))^{q-p} (a^{-2}b^2)b (b(a^{-2}b^2) b)^{p-1}=1.\] Now since $b,(a^{-2}b^2)$ are negative, it follows that $b^{-1}ab>_\ooo 1$ and then $b^{-1}a^{-1}b<_\ooo 1$ and $b^{-1}a^{-2}b<_\ooo 1$. But now the relator $(\mu\lambda)^{q-p} (\mu^2\lambda)^p=1$ can be rewritten again as \[(b^2(b^{-1}a^{-2}b)a)^{q-p}(b^2(b^{-1}a^{-2}b)b^2)^p=1,\] where all the terms on the left-hand side are positive and the right-hand side is the identity. We arrive at a contradiction. Hence, we must have $a^{-2}b^2>_\ooo 1$, that is, $b^{-2}a^2<_\ooo 1$. As in the previous case, when $p=q=1$ then the relator $b(a^{-2}b^2)b = 1$ forces $b^{-2}a^2 <_\ooo 1$.

        Since $\mu^3\lambda=ab^{-3}a^2 =(ab^{-1})(b^{-2}a^2)>_\ooo 1$, we must have $ab^{-1}>_\ooo 1$. Rewrite the relator $(\mu\lambda)^{q-p} (\mu^2\lambda)^p=1$ one more time as \[ ((ba^{-1})(a^{-1}b^2)(b^{-1}a))^{q-p}((ba^{-1})(a^{-1}b^2)b)^p =1.\]
        Since $ba^{-1},b^{-1}a$ and $b$ are all negative, we must have $a^{-1}b^2>_\ooo 1$ and so $b^{-2}a<_\ooo 1$. But then \[m^2l=ab^{-3}ab=(ab^{-1})(b^{-2}a)b<_\ooo 1,\] contradicting $m^3l>_\ooo 1$ and $m<_\ooo 1$.
    \end{casesp} \end{proof}

\begin{corollary}
    \label{cor:whcor}
    Suppose $p,q$ are coprime positive integers with $p/q \in (1, \infty)$. If $[\beta] \in (1, \infty) \subset \mathcal{S}(S^3_{*, \frac{p}{q}}(\wh))$, then $[\beta]$ is not weakly order-detected.
\end{corollary}
\begin{proof}
    Using $\bar{m}, \bar{l} \in \pi_1(S^3_{*, \frac{p}{q}}(\wh))$ to denote the image of the peripheral elements $m, l \in \pi_1(M)$, Lemma \ref{lem:whnormalclosure} implies that $N(\bar{m}\bar{l}) \cap N(\bar{m}) \neq \emptyset$.

    If $p/q \in (1,2]$, combining this with Proposition \ref{prop:whprop} we may apply Theorem \ref{thm:cofinaltrick} to conclude that no $[\beta] \in (1, \infty)$ is weakly order-detected.

    On the other hand suppose $p/q \in (2, \infty)$ and that $[\alpha] \in (2, \infty) \subset \mathcal{S}(S^3_{*, \frac{p}{q}}(\wh))$ is weakly order-detected. Then we can use the short exact sequence
    \[ \{1\} \longrightarrow \inner{\inner{\mu^p \lambda^q}} \longrightarrow \pi_1(M) \longrightarrow \pi_1(S^3_{*, \frac{p}{q}}(\wh)) \longrightarrow \{1\}\]
    to argue that
    $(\emptyset, \emptyset ; [\alpha], [\mu^p \lambda^q])$ is order-detected, contradicting Theorem \ref{thm:lnnnotweak} in the case $n=0$. We can now use Theorem \ref{thm:cofinaltrick} to conclude that no $[\beta] \in (1, \infty) \subset \mathcal{S}(S^3_{*, \frac{p}{q}}(\wh))$ is weakly order-detected.
\end{proof}

\begin{proposition}
    Suppose $(p_1/q_1, p_2/q_2) \in (1, \infty) \times (1, \infty)$ where $p_i, q_i$ are coprime positive integers. Then $(\{2\}, \{2\}; [m^{p_1}l^{q_2}], [\mu^{p_2}\lambda^{q_2}])$ is not order-detected by a left-ordering $\ooo$ of $\pi_1(M)$.
\end{proposition}
\begin{proof}
    For contradiction, suppose that $(\{2\}, \{2\}; [m^{p_1}l^{q_2}], [\mu^{p_2}, \lambda^{q_2}])$ is order-detected by a left-ordering $\ooo$ of $\pi_1(M)$. Then there exists an $\ooo$-convex normal subgroup $C \subset \pi_1(M)$ such that $C \cap \pi_1(T_2) = \inner{\mu^{p_2} \lambda^{q_2}}$ and $C \cap \pi_1(T) \subset \inner{m^{p_1}l^{q_1}}$. In particular, $\inner{\inner{\mu^{p_2} \lambda^{q_2}}} \subset C$, so there exists a homomorphism
    \[ \phi: \pi_1(S^3_{*, \frac{p_2}{q_2}}(\wh)) \cong \pi_1(M)/\inner{\inner{\mu^{p_2} \lambda^{q_2}}} \longrightarrow \pi_1(M)/C.
    \]

    Set $K = \ker(\phi)$ and consider the short exact sequence
    \[ \{1\} \longrightarrow K \longrightarrow \pi_1(S^3_{*, \frac{p_2}{q_2}}(\wh)) \stackrel{\phi}{\longrightarrow} \pi_1(M)/C \longrightarrow \{1\}.\]
    Equip $\pi_1(M)/C$ with the quotient left-ordering $\ooo'$ defined by $gC <_{\ooo'} hC$ if and only if $g <_{\ooo} h$ whenever $gC \neq hC$, and equip $K$ with any left-ordering whatsoever. Using these left-orderings, construct a lexicographic left-ordering $\ooo''$ of $\pi_1(S^3_{*, \frac{p_2}{q_2}}(\wh))$.

    Note that for all $m^r l^s \in \pi_1(T_1) \setminus \inner{m^{p_1}l^{q_1}}$ we have $ m^r l^s >_{\ooo} 1$ if and only if $m^r l^sC >_{\ooo'} C$. Therefore, if we denote the images of $m, l$ in $\pi_1(S^3_{*, \frac{p_2}{q_2}}(\wh)) $ by $\bar{m}, \bar{l}$, then we have $ \bar{m}^r \bar{l}^s >_{\ooo''} 1$ if and only if $ m^r l^s >_{\ooo} 1$. Thus $\ooo''$ weakly order-detects $[m^{p_1}l^{q_1}] \in \mathcal{S}(S^3_{*, \frac{p_2}{q_2}}(\wh))$. This contradicts Corollary \ref{cor:whcor}.
\end{proof}

We require the next remark for our final proof of this section.

\begin{remark}
    \label{rem:automorphism}
    Note that there is an automorphism $f: \pi_1(M) \rightarrow \pi_1(M)$ given by $f(a) = a^{-1}b^3$ and $f(b) = a^{-1}ba$, and this automorphism satisfies \[f(m) = \mu,\,f(l) = \lambda,\,f(\mu)=m\mbox{ and }f(\lambda)=l.\] Therefore if we let $\sigma:\{1,2\} \rightarrow \{1,2\}$ denote the transposition $\sigma(1)=2$ and $\sigma(2)=1$, then $(J, K; [\alpha_1], [\alpha_2])$ is order-detected if and only if $(\sigma(J), \sigma(K); f_*([\alpha_1]), f_*([\alpha_2]))$ is order-detected. Here we have used $f_*$ to denote the induced map $f_*:\mathcal{S}(T_1) \rightarrow \mathcal{S}(T_2)$.
\end{remark}

% \begin{theorem}
%     If $(r_1, r_2) \in [1, \infty) \times [1, \infty)$ is a pair rational numbers, then $\pi_1(S^3_{r_1, r_2}(\wh))$ is not left-orderable.
% \end{theorem}
\begin{proof}[Proof of Theorem \ref{thm:introwhitehead}]
    Note that $S^3_{1,*}(\wh)$ is the trefoil knot complement, so the conclusion holds true if $r_1=1$ or $r_2=1$, since Dehn fillings of the trefoil yield non-left-orderable fundamental groups when the filling slope is greater than or equal to one.

    Now suppose that $(r_1, r_2) = (p_1/q_1, p_2/q_2) \in (1, \infty) \times (1, \infty)$, where $p_i/q_i$ is written in lowest terms, and that $\pi_1(S^3_{r_1, r_2}(\wh))$ is left-orderable. Then we apply Remark \ref{rem:sub-detection} and Theorem \ref{thm:loimplieswkdet} to conclude that one of $(\{1\}, \{1\}; [m^{p_1} l^{q_1}], [\mu^{p_2} \lambda^{q_2}])$ or $(\{2\}, \{2\}; [m^{p_1} l^{q_1}], [\mu^{p_2} \lambda^{q_2}])$ is order-detected. However, $(\{2\}, \{2\}; [m^{p_1} l^{q_1}], [\mu^{p_2} \lambda^{q_2}])$ is not order-detected as this would contradict Proposition \ref{prop:whprop}.

    On the other hand, if $(\{1\}, \{1\}; [m^{p_1} l^{q_1}], [\mu^{p_2} \lambda^{q_2}])$ is order-detected then, then an application of Remark \ref{rem:automorphism} shows that this case also contradicts Proposition \ref{prop:whprop}.
\end{proof}

\subsection{Cofinal orderings of the figure-eight knot group}

The results of the previous section also carry consequences for the manifolds $S^3_{r, *}(\wh)$ and $S^3_{*, r}(\wh)$. We illustrate these ideas by considering the case of $S^3_{*,-1}(\wh)$, the figure-eight knot complement. We begin by recording a lemma.

\begin{lemma}
    \label{lem:cofinalinterval}
    Let $M$ denote the complement of $\wh$ in $S^3$, with notation as above, and let $\ooo$ be a left-ordering on $\pi_1(M)$.
    \begin{enumerate}
        \item If $[\alpha_1] \in (1,\infty)$ and $(\emptyset, \emptyset; [\alpha_1], [\alpha_2])$ is order-detected by $\ooo$, then $\pi_1(T_1)$ is not $\ooo$-cofinal.
        \item If $[\alpha_2] \in (1,\infty)$ and $(\emptyset, \emptyset; [\alpha_1], [\alpha_2])$ is order-detected by $\ooo$, then $\pi_1(T_2)$ is not $\ooo$-cofinal.
    \end{enumerate}
\end{lemma}
\begin{proof}
    To prove (1), suppose $[\alpha_1] \in (1,\infty)$ and $(\emptyset, \emptyset; [\alpha_1], [\alpha_2])$ is order-detected by $\ooo$, and $\pi_1(T_1)$ is $\ooo$-cofinal. By Lemma \ref{lem:whnormalclosure}, $ml$ and $m$ cannot both be cofinal and have opposite signs, so this is a contradiction. To arrive at (2), we apply the automorphism $f: \pi_1(M) \rightarrow \pi_1(M)$ appearing in Remark \ref{rem:automorphism}.
\end{proof}

Recall that if $K$ is the figure-eight knot with $M_K$ being its complement, then
\[ \pi_1(M_K) = \langle x, y \mid xy^{-1}x^{-1}yx=yxy^{-1}x^{-1}y \rangle\]
with meridian and longitude $\mu_K = x$ and $\lambda_K = yx^{-1}y^{-1}x^2y^{-1}x^{-1}y$ that generate the peripheral subgroup $\pi_1(\partial M_K)$. %We identify the slopes $\mathcal{S}(\partial M_K)$ with $\mathbb{R} \cup \{ \infty \}$ in the usual way, so that $[\mu^p \lambda^q]$ is identified with $p/q \in \mathbb{R}$. 
There is a quotient homomorphism (arising from Dehn filling the first component of the mirror of the Whitehead link) $\psi : \pi_1(M) \rightarrow \pi_1(M_K)$ determined by
%todo{On my original notes, it says something different. I will check again}
\[\psi(a) = x^2y^{-1}x \mbox{ and } \psi(b) = x^2y^{-1}.\]
One checks that $\psi(m) = \mu_K$ and $\psi(l) = \lambda_K^{-1}$, so that $\psi(\pi_1(T_1)) = \pi_1(\partial M_K)$. One also checks that $\psi(\mu) = xy^{-1}$, so that $\psi(\pi_1(T_2)) = \langle xy^{-1} \rangle$. There exists an outer automorphism $\phi: \pi_1(M_K) \rightarrow \pi_1(M_K)$ determined by
\[ \phi(x) = x \mbox{ and } \phi(y) = x^{-1}yxy^{-1}x,\]
which arises from the fact that the figure-eight knot is amphichiral. We see that $\phi(\mu_K) = \mu_K$, and $\phi(\lambda_K) = \lambda^{-1}_K$, so that $\phi(\pi_1(\partial M_K)) = \pi_1(\partial M_K)$.

\begin{proposition}
    \label{prop:fig8}
    Suppose that $\mathfrak{o}$ is a left-ordering of $\pi_1(M_K)$.
    \begin{enumerate}
        \item If $\pi_1(\partial M_K)$ is $\ooo$-cofinal, then $s(\ooo) \in [-1, 1] \cup \{ \infty \}$.
        \item If $s(\ooo) \in (-\infty, 2) \cup (2, \infty)$, then $\langle xy^{-1} \rangle$ is $\ooo$-cofinal.
    \end{enumerate}
\end{proposition}
\begin{proof}
    Consider the short exact sequence
    \[\{1\}\longrightarrow \langle \langle \mu^{-1} \lambda \rangle \rangle \longrightarrow \pi_1(M) \stackrel{\psi}{\longrightarrow} \pi_1(M_K) \longrightarrow\{1\}\]
    where $M$ is the complement of $\wh$ in $S^3$.

    To prove (1), suppose that $\ooo'$ is a left-ordering of $\pi_1(M_K)$ with $s(\ooo') \in (-\infty, -1) \cup (1, \infty)$. If $\ooo$ is a lexicographic ordering of $\pi_1(M)$ constructed relative to the short exact sequence above using $\ooo'$ as the left-ordering of $\pi_1(M_K)$, then we see that $\ooo$ order-detects $(\emptyset, \{2\}; [\alpha], [\mu^{-1} \lambda])$ where $[\alpha] \in (-\infty, -1) \cup (1, \infty)$.

    If $[\alpha] \in (1, \infty)$ then $\pi_1(T_1)$ is not $\ooo$-cofinal, by Lemma \ref{lem:cofinalinterval}, and so $\pi_1(\partial M_K)$ is not $\ooo'$-cofinal. On the other hand, if $[\alpha] \in (-\infty, -1)$ then applying the automorphism $\phi$ to $\ooo'$ yields a left-ordering $\ooo''$ of $\pi_1(M_K)$ order-detecting the slope $-[\alpha] \in (1, \infty)$. Applying Lemma \ref{lem:cofinalinterval} to $\ooo''$, we conclude that $\pi_1(\partial M_K)$ is not $\ooo''$-cofinal, and thus not $\ooo'$-cofinal, either. This proves (1).

    To prove (2), suppose $\ooo'$ is a left-ordering of $\pi_1(M_K)$ order-detecting $[\alpha] \in (2, \infty)$. Then if $\ooo$ is a lexicographic ordering of $\pi_1(M)$ constructed relative to the short exact sequence above using $\ooo'$ as the left-ordering of $\pi_1(M_K)$, we see that $\ooo$ order-detects $(\emptyset, \{1, 2\}; [\alpha], [\mu^{-1} \lambda])$ with $[\alpha] \in (2, \infty)$. Then $\pi_1(T_2)$ must be $\ooo$-cofinal by Corollary \ref{cor:bddrangeLn}(2) (where we take $n=0$), and so $\psi(\pi_1(T_2)) = \langle xy^{-1} \rangle$ must be $\ooo'$-cofinal as well.
    %Similar to the previous paragraph, we may use the automorphism $\phi$ to conclude that $\langle y \lambda^{-1}yx^{-1} \rangle$ also $\ooo'$-cofinal.
\end{proof}

\begin{remark}
    By \cite{BGYpreprinta} the slopes $[\mu_K^{\pm 1}\lambda_K]$ and $[\mu_K]$ are order-detected by left-orderings of $\pi_1(M_K)$ relative to which $\pi_1(\partial M_K)$ is cofinal. This implies that the intervals $(-\infty, -1)$ and $(1, \infty)$, which are complementary to the set $[-1, 1] \cup \{ \infty \}$ appearing in Proposition \ref{prop:fig8}(1), cannot be ``enlarged'' by any improvement in our computations.
\end{remark}

\bibliographystyle{alpha}

\bibliography{whitehead}

\end{document}